\documentclass{amsart}

\usepackage{amssymb}
\usepackage{enumerate}

\theoremstyle{plain}
 \newtheorem{thm}{Theorem}[section]
 \newtheorem{prop}[thm]{Proposition}
 \newtheorem{lem}[thm]{Lemma}
 \newtheorem{cor}[thm]{Corollary}
 \newtheorem{fact}[thm]{Fact}
 \newtheorem{teorema}{Theorem}

\theoremstyle{definition}
 \newtheorem{dfn}[thm]{Definition}
 \newtheorem{que}[thm]{Question}
 
 \newtheorem{exm}[thm]{Example} 
 \newtheorem{rmk}[thm]{Remark}

\def\strok{\!\upharpoonright\!}
\def\forces{\vdash}
\def\Mon{{\mathbb U}}
\renewcommand{\leq}{\leqslant}

\DeclareMathOperator{\ior}{\vdash\hspace*{-0.1em}^{\mathit i}}

\DeclareMathOperator{\tp}{tp}
\DeclareMathOperator{\itp}{itp}
\DeclareMathOperator{\cl}{cl}
\DeclareMathOperator{\Aut}{Aut}
\def\ccel{{\it ccel}}

\begin{document} 

 \title{Around Rubin's ``Theories of linear order''}
  
\author{Predrag Tanovi\'c}
\address[P.\ Tanovi\'c]{Mathematical Institute SANU\\ Knez Mihailova 36, Belgrade, Serbia}
\email[P.\ Tanovi\'c]{tane@mi.sanu.ac.rs}
\thanks{The first author was supported by the Ministry of Education, Science and Technological Development of Serbia through Mathematical Institute SANU}
 
\author{Slavko Moconja}
\address[S.\ Moconja]{University of Belgrade, Faculty of Mathematics\\ Studentski trg 16, 11000 Belgrade, Serbia \and Instytut Matematyczny, Uniwersytet Wroc\l{}awski\\ pl.\ Grunwaldzki 2/4, 50-384 Wroc\l{}aw, Poland}
\email[S.\ Moconja]{slavko@matf.bg.ac.rs}
\thanks{The second author was supported by the Narodowe Centrum Nauki grant no.\ 2016/22/E/ST1/00450, and by the Ministry of Education, Science and Technological Development of Serbia through University of Belgrade, Faculty of Mathematics}

\author{Dejan Ili\'c}
\address[D.\ Ili\'c]{University of Belgrade, Faculty of Transport and Traffic Engineering\\ Vojvode Stepe 305, 11000 Belgrade, Serbia}
\email[D.\ Ili\'c]{d.ilic@sf.bg.ac.rs}
\thanks{The third author was supported by the Ministry of Education, Science and Technological Development of Serbia through University of Belgrade, Faculty of Transport and Traffic Engineering}


 \begin{abstract} 
Let $\mathcal M=(M,<,...)$ be a linearly ordered first-order structure and $T$ its complete theory. We investigate conditions for $T$ that could guarantee that  $\mathcal M$ is not much more complex than some colored orders (linear orders with added unary predicates). Motivated by Rubin's work \cite{Rubin}, we label three conditions expressing properties of types of $T$ and/or automorphisms of models of $T$. 
We prove several results which indicate the ``geometric'' simplicity of definable sets in models of theories satisfying these conditions. For example, we prove that the strongest condition characterizes, up to definitional equivalence (inter-definability), theories of colored orders expanded by equivalence relations with convex classes. \end{abstract} 

 \maketitle

In 1973, in his master's thesis and the derived paper \cite{Rubin}, Matatyahu Rubin developed powerful  techniques for analyzing model-theoretic  properties of complete, first-order  theories of infinite colored orders (linear orders with added unary predicates). Rubin investigated  finite axiomatizability, topological complexity of type spaces and  saturation of their models, and proved several deep results, the most renowned  being  that the  number of countable models of such theory $T$ is either continuum or finite  (even equal to 1 in the finite language case). Rubin's proof had not been modified before 2015, when Richard Rast in \cite{Rast} improved this result by classifying the isomorphism relation for countable models of $T$ up to Borel bi-reducibility.   
Only recently, the first two authors in \cite{MTso} have generalized Rubin's theorem to a substantially wider context of binary, stationarily ordered, first-order theories. 

  Although  most of Rubin's analysis  uses only basic model-theoretic tools (the compactness theorem and Ehrenfeucht games),  its
deepest part has a strong topological flavor, wherein
some general topology methods are used both as a tool and a way of expressing properties of colored orders. 
In this article, we will continue Rubin's work in the geometric model theory direction. In addition to finding  a ``geometric'' description of definable sets in colored orders, our  main goal is determining wider classes 
of linearly ordered structures whose definable sets are not much more complex than those in colored orders. By a linearly ordered structure $(M,<,...)$ we mean a first-order structure in any language  containing the symbol $<$ interpreted as a linear order. We will consider three  classes of structures and the largest   among them is that of {\it linearly finite} structures. A linearly ordered  structure  $(M,<,...)$  is linearly finite, if:

\smallskip
\noindent
{\bf (LF)} \ For every partitioned formula $\phi(\bar x;\bar y)$ there is an integer $n_{\phi}$ such that for every initial part $C\subset M$  at most $n_{\phi}$  complete $\phi$-types with parameters from $C$   are realized in $M\smallsetminus C$. 

\smallskip
Condition (LF)  is preserved under the elementary equivalence of structures. Therefore, it is a property of the complete first-order   theory of the structure, in which case we say that the theory has the property (LF). The other two classes are   also classes of all models of the theories satisfying certain conditions. 
A complete theory $T$ of linearly ordered structures   satisfies   {\it the strong linear binarity  condition} if:

\smallskip\noindent
{\bf (SLB)} For every  model  $(M,<,...)\models T$,  initial part $C\subset M$  and automorphism $f\in \Aut(M)$  fixing $C$ setwise, the mapping defined by: $g(x)=f(x)$ for $x\in C$ and $g(x)=x$ for $x\notin C$, is an automorphism of $M$.

\smallskip
The third condition is  the  {\it  linear binarity condition}, denoted by (LB). It can be found in Rubin's paper stated in a topological form as a property of type-spaces of complete theories of colored orders. In fact, (LB) motivated us to introduce the other two conditions: (LB) is equivalent to the weak form of (SLB) obtained by referring only to initial intervals   $C=(-\infty,a]$  (for all $a\in M$). 
We will prove  that, in general,  (LB)  is strictly weaker than (SLB) and strictly stronger than (LF). We will also show that (SLB) is satisfied not only in colored orders  but  in \ccel-orders (colored orders expanded by  equivalence relations with convex classes), too. All these are explained in detail in Section 2.

Our main result is a complete characterization of theories satisfying (SLB). We will prove that,
 up to definitional equivalence,  they are theories of \ccel-orders, and we will offer a precise geometric description of  definable sets in these structures. 
In order to explain the above in more detail, let $\mathcal M=(M,<,...)$ be a linearly  ordered $L$-structure and let $T$ be its complete theory. The {\it \ccel-reduct of 
$\mathcal M$} is the structure $\mathcal M^{ccel}=(M,<,P_i,E_j)_{i\in I,\,j\in J}$ in which all   unary definable sets $P_i$ and all definable convex equivalence relations $E_j$ are named; here we assume that the underlying language $L_T$ is chosen in some uniform  way modulo $L$ and $T$. 
The complete $L_T$-theory of $\mathcal M^{\ccel}$ does not depend on the choice of $\mathcal M\models T$; call this theory  {\it the \ccel-companion of $T$} and denote it by $T^{ccel}$. Recall that two first-order structures are definitionally equivalent if they have the same domain and the same definable sets (of tuples).  

\begin{teorema}\label{teorema1}
A complete theory of linearly ordered structures $T$ satisfies the strong linear binarity condition if and and only if it is definitionally equivalent with $T^{ccel}$:  some (equivalently any) model $\mathcal M\models T$ is definitionally equivalent with $\mathcal M^{ccel}$. 
\end{teorema}

In general, theories of \ccel-orders do not eliminate quantifiers: if $n$ is an integer and $E$ is a  convex equivalence relation, then the relation ``$x$ is in the $n$-th successor/predecessor $E$-class of the $E$-class of $y$'', denoted by $x\in S^n_E(y)$, is not necessarily expressible by a quantifier-free formula. We will prove that the lack of these successor-relations is essentially the only obstruction for the elimination of quantifiers. 
By a {\it $u$-convex} formula we will mean either a unary formula or a formula $\theta(x,y)$ which is the conjunction   of a unary $L$-formula $\psi(x)$ and one of the following formulae:\footnote{They are precisely defined in Definition \ref{Definition S^n_E}.}
\begin{center}
$S^{-m}_{E_1}(y)\leqslant x<S^{-n}_{E_2}(y)$, \ $S^{-m}_{E_1}(y)\leqslant x\leqslant S^{n}_{E_2}(y)$  \ and \  $S^{m}_{E_1}(y)< x\leqslant S^{n}_{E_2}(y)$
\end{center}
where $E_1$ and  $E_2$ are definable convex equivalence relations and $m,n$ non-negative integers. The above mentioned geometric description is:

\begin{teorema}\label{teorema2}
If   $T$ satisfies the strong linear binarity condition, then  every $L$-formula is  equivalent modulo $T$ to a Boolean combination of $u$-convex formulae.  
\end{teorema}
 
From a model-theoretic point of view, colored orders are as complicated as pure linear orders are: it is well known that for a given linear order with finitely many colors there is a canonical way  of producing a pure linear order in which the original structure is interpretable. Similarly, \ccel-orders  (in a finite language) are canonically interpretable in  colored orders and hence in pure linear orders, too. Therefore, definable sets in both pure linear orders, colored orders and \ccel-orders  are equally complex.  Theorem \ref{teorema2} provides a geometric description of definable sets in these structures, and even in the case of colored orders, the description  is somewhat more comprehensive than  Simon's from \cite{Simon}.  

\medskip 
In most of the paper, we will deal with linearly finite structures.
We will find fairly precise descriptions of their parametrically definable convex sets and their
unary definable functions. We will also describe 
one-parameter definable subsets of these structures. Then the main results, as well as a not-so-precise description of all parametrically definable subsets of (LB)-structures, will follow rather routinely.

\smallskip 
The paper is organized as follows: Section 1 contains preliminaries and basic facts. In Section 2 we   introduce   the three conditions and study their relationship. In Section 3 we study convex sets definable in linearly finite structures. In Section 4 we introduce almost convex equivalence relations as finite-index refinements of convex equivalence relations. We prove that every one-parameter definable set  in a linearly finite structure is a Boolean combination of unary definable sets, intervals, and classes of almost convex equivalence relations; the same description applies to all parametrically definable sets assuming (LB). In Section 5, rather as corollaries of the previous results, we deduce the main ones.

\section{Preliminaries}

Throughout the paper, $L$ is a first-order language containing a binary relation symbol $<$,  $T$ is a complete $L$-theory having infinite models in which $<$ defines a linear order (denoted by the same symbol) and  $(\Mon,<,...)$ is a very large sufficiently saturated and strongly  homogeneous model of $T$. By $a,b,\dots $ we will denote its elements, by $\bar a,\bar b,\dots $ tuples of elements, and by $A,B,A',\dots $ its small (smaller than the degree of saturation of $\Mon$)  subsets. Letters $C,D,\dots $ are reserved for subsets  which are not necessarily small, e.g.\ for definable sets. 
$S_n(T)$ is the space of all complete types in $n$ variables $\bar x$ with   basic clopen sets   of the form   $[\phi]=\{p\in S_n(T)\mid \phi(\bar x)\in p\}$ for formulae $\phi(\bar x)$. 
For $C$ a subset of the domain of an $L$-structure, by $\phi(C^n)$ we will denote the solution set of $\phi(\bar x)$ in $C^n$. By  definable subsets (relations,\dots ) of the $L$-structure we will mean $L$-definable ones; similarly for type-definable subsets, i.e.\ subsets defined by an infinite set (conjunction)  of formulae.  $L(A)$ is the language $L$  expanded by constants for elements of $A$ and solution sets  of $L(A)$-formulae are $A$-definable sets.

Let $\phi(\bar x;\bar y)$ be a partitioned formula (one in which variables $\bar y$ are reserved for parameters), let $|\bar x|=n$ and $|\bar y|=m$. By a $\phi$-type over parameters $A$ we will mean a  consistent set of formulae $\{\phi^{\epsilon(\bar a)}(\bar x,\bar a)\mid\bar a\in A^m\}$  where $\epsilon\in 2^{A^m}$, $\phi^0:=\lnot\phi$ and  $\phi^1:=\phi$. For a tuple $\bar c\in\Mon^n$, we write $\tp_\phi(\bar c/A)$ for the $\phi$-type of $\bar c$ over $A$: $\tp_\phi(\bar c/A)=\{\phi^\epsilon(\bar x,\bar a)\mid \bar a\in A^m,\epsilon\in 2\mbox{ and }\models\phi^\epsilon(\bar c,\bar a)\}$. If $\Pi(x)$ is a partial type over $A$, then by $\Pi(\Mon)$ we denote its locus, i.e. the solution set. 
Two first-order structures  (in possibly distinct languages) are {\it definitionally equivalent} (or inter-definable) if they have the same domain and the same definable sets. Two complete theories are definitionally equivalent if they have  (a pair of) definitionally equivalent models.  

We write $C<D$ if $c<d$ holds for all $c\in C$ and $d\in D$; similarly for $C\leqslant D$ and  $c<D$. $C$ is an {\em initial part} if $c\in C$ and  $d<c$ imply $d\in C$; a {\em final part} is defined dually. Intervals
$(a,\infty)$ and $(-\infty,a)$ are defined in the usual way. An element $d$ is an {\em upper bound} of $C$ if $C<d$ holds;  {\em lower bounds} are defined dually.
By $\sup(C)<\sup(D)$ we mean that the set of upper bounds of $D$ is strictly contained in the set of upper bounds of $C$; similarly for $\sup(C)\leqslant \sup( D)$, $a<\sup D$, etc.  If $C,D$ are definable sets, then these relations are definable, and we will use the same notation for the defining formulae.
 A {\it convex equivalence  relation} on a linearly  ordered  set  is one whose classes are $<$-convex subsets.

\subsection{Interval types}

A  formula in one free variable is {\em convex (initial, final)} if it defines a convex (initial, final) subset of $\Mon$. 
By  {\em an interval type over $A$} we mean a maximal partial 1-type $\Pi(x)$ consisting of convex $L(A)$-formulae. An interval type $\Pi$ is closed for finite conjunctions, and for every convex $L(A)$-formula $\phi(x)$,  either $\phi(x)\in\Pi(x)$ or $\phi(x)$ is inconsistent with $\Pi(x)$. Some authors, notably Rosenstein in \cite{Rosenstein}, define interval types as maximal partial 1-types consisting of initial and final formulae. The two definitions are not ambiguous  since every interval type in Rosenstein sense uniquely extends to an interval type in the sense of the definition given here.
The set of all interval types over $A$ is denoted by $IT(A)$. It is endowed with compact, Hausdorff topology in the usual way.
Clearly, the locus of an  interval type is a convex set, and distinct interval types over $A$ have disjoint loci   so that $IT(A)$ is naturally linearly ordered by $<$.

{\it The interval type of $a$ over $A$}, denoted by $\itp(a/A)$,  consists of all convex $L(A)$-formulae satisfied by $a$.   It is easily seen that the locus of $\itp(a/A)$ is the convex hull of the locus of $\tp(a/A)$.

For $A\subseteq B$ and $\Pi\in IT(A)$ we will write $\Pi\ior \Pi\,|\,B$  if there is a unique interval type over $B$ extending $\Pi$ (denoted by $\Pi\,|\,B$). We will say that an initial part $D$ {\em cuts} (the locus of) an interval type $\Pi$ if there exist $a,b\models\Pi$ such that $a\in D$ and $D<b$. Note that if $\Pi\in IT(A)$, then such an  $a$ and $b$ can be chosen having the same type over $A$, because  $\Pi(\Mon)$ is the convex hull of the locus of $\tp(a/A)$.

\begin{fact}\label{lema itp on cutting} 
\begin{enumerate}[(a)]
\item Let $\Pi\in IT(A)$ and $A\subseteq B$. Then $\Pi\ior\,\Pi\,|\, B$ if and only if no $B$-definable initial part cuts $\Pi$.
\item For every parametrically definable  set $D\neq\emptyset$ and every $A$ there exists a maximal (with respect to $<$) interval type from $IT(A)$  consistent with $x\in D$.
\item A parametrically definable initial part does not cut any interval type from $IT(A)$ if and only if it is $A$-definable.
\end{enumerate}
\end{fact}
\begin{proof} (a) This is easy.

\smallskip
(b) Toward contradiction, assume that the set $S$ of all interval types over $A$ consistent with $x\in D$ does not have maximum. By compactness, for each $\Pi\in S$ there is an initial formula $\phi_{\Pi}(x)\in \Pi(x)$  such that 
$\lnot\phi_\Pi(x)\land x\in D$ is consistent. 
The set  $\{\lnot\phi_\Pi(x)\mid \Pi\in S\}\cup\{x\in D\}$ is also consistent so, by compactness and saturation, there exists $a\in D$ realizing it. Then $a\in D$ implies $\itp(a/A)=\Pi$ for some $\Pi\in S$ which is in contradiction with  $\lnot\phi_\Pi(x)\in\tp(a)$.

\smallskip
(c) The right-to-left implication is clear. To prove the other one, assume that $D$ is a parametrically definable  initial part not cutting any interval type from $IT(A)$. By part (b), there is a  maximal interval type  $\Pi\in IT(A)$ consistent with $x\in D$. Since $D$ is an initial part, does not cut $\Pi$ and intersects $\Pi(\Mon)$,  we have $\Pi(x)\forces x\in D$. By compactness there exists $\phi(x)\in\Pi$ with $\phi(x)\forces x\in D$. 
We  claim that $D$ is defined by the $L(A)$-formula $\exists y\,(\phi(y)\land x\leq y)$.
The solution set $D_0$ of this formula is an initial part contained in $D$ because $D$ is initial and $\phi(x)\forces x\in D$. On the other hand, if $D_0\subset D$ were true, then any $a\in D\smallsetminus D_0$ would satisfy $\phi(\Mon)<a$ and $\Pi<\itp(a/A)$; the latter contradicts the maximality of $\Pi$. Hence $D_0=D$ and $D$ is $A$-definable. 
\end{proof}

\section{Binarity and binarity-like conditions}\label{S binarity}

 Bruno Poizat in his book \cite{Poizat} extracted the following definition from Rubin's work:
Given a sequence $(\phi_1(x), \dots , \phi_n(x))$ of formulae in one free
variable $x$ and two elements $a < b$,
we will say that this sequence of formulae is realized between $a$ and $b$ if
there are  elements  $c_1,c_2,\dots ,c_n$ with $c_i$ satisfying $\phi_i (x)$ for $i\leqslant n$ 
and such that $a<c_1<\dots <c_n<b$. Poizat's Theorem 12.32, which he calls Rubin's Theorem, is a restatement of  Rubin's Corollary 3.9 and states that theories of colored orders have the following property:

\medskip
\noindent{\bf (RB)} \  Two increasing $n$-tuples  $a_1<\dots <a_n$ and $b_1<\dots <b_n$ of elements of models of $T$  have the same
type if and only if they satisfy the following conditions:
\begin{itemize}
\item  $\tp(a_i)=\tp(b_i)$  for every $i\leqslant n$;
\item For every $i< n$ the same finite sequences of formulae are realized
between $a_i$ and $a_{i+1}$  as between $b_i$ and $b_{i+1}$. 
\end{itemize}
(Note that here it suffices to verify that $n$-tuples of some $\omega$-saturated model of $T$  satisfy the above condition.)

We have chosen (RB) as an abbreviation for Rubin's binarity condition; to justify the word binarity here, first recall that a first-order theory $T$ is {\it binary} if every formula is equivalent  modulo $T$ to a Boolean combination of formulae in at most two free variables; equivalently: every complete $n$-type is determined by the union of all its complete 2-subtypes. It is easy to see that the complete theory of any $\omega$-saturated linearly ordered structure  satisfying condition (RB) is binary.  In particular, Rubin's Theorem implies that any complete theory of  colored orders is binary.   
We will prove a little bit more: that (RB) is the key property of theories of colored orders, i.e.\ that it characterizes, up to definitional equivalence, colored orders among the saturated linearly ordered structures.

\begin{thm}\label{Prop RB}
An $\omega$-saturated, linearly ordered structure   satisfies condition (RB) if and only if it is definitionally equivalent to a  colored order. 
\end{thm} 
\begin{proof}The right-to left direction is Rubin's theorem. To prove the other direction, suppose that a  $(\Mon,<,\dots )$   satisfies (RB); then its complete theory is binary.
  Let   $L_u$ be the language consisting  of  $<$ and   (new)   symbols  for all unary  $L$-definable subsets.   Interpret the new symbols naturally to obtain an  $L\cup L_u$-expansion; it is a saturated, definitional expansion of the original structure. In order to prove that its $L$-reduct and its $L_u$-reduct are definitionally equivalent, and having on mind the binarity, it suffices to show that every  binary $L$-formula is equivalent (modulo the theory $T^*$ of the expansion) to an $L_u$-formula.
Let  $\mathcal F$ be the set of all finite sequences of unary $L$-formulae. For each finite sequence   $\vec \phi=(\phi_1(x),\dots , \phi_n(x))\in \mathcal F$  define: 
  $$\theta_{\vec\phi}(x,y):= \exists z_1\dots z_n\,(x<z_1<\dots <z_n<y\land \bigwedge_{1\leqslant i\leqslant n}\phi_i(z_i)).$$ 
$\theta_{\vec\phi}(x,y)$ is an $L$-formula describing that the sequence $\vec\phi$ is realized between $x$ and $y$; note that it is $T^*$-equivalent to an $L_u$-formula. 
   
  Let $\psi(x,y)$ be a consistent $L$-formula implying $x<y$. For each $L$-type $r(x,y)\in S_2(T)$ containing $\psi(x,y)$   choose a formula $\sigma_r(x,y)\in r$ implying $\psi(x,y)$ in the following way:
 Let $(a,b)$ realize $r$. By applying condition (RB) to   $r(x,y)$ we have:
 $$ \tp_x(a)\cup \tp_y(b)\cup \{x<y\}\cup\{\theta_{i}^{\epsilon_{i}}(x,y)\mid i\in\mathcal F, \ \epsilon_i\in\{0,1\} \mbox{ and } \models \theta_{i}^{\epsilon_{i}}(a,b)\}\vdash \psi(x,y).$$
By compactness,  there are $\chi_1(x)\in \tp_x(a)$, $\chi_2(y)\in\tp_y(b)$, and  a finite subset $\mathcal F_0\subseteq \mathcal F$ with:   
$$ \models \left(\chi_1(x)\land \chi_2(y)\land x<y\land \bigwedge_{i\in \mathcal F_0}\theta_i^{\epsilon_i}(x,y)\right)\rightarrow \psi(x,y).$$
Denote by $\sigma_r(x,y)\in r$ the formula   on  the left hand side of the implication.

 Now, we have  a cover $\{[\sigma_r]\mid r\in [\psi]\subseteq S_2(T)\}$  of   the closed subset  $[\psi]$ of $S_2(T)$. By compactness, there is a finite subcover. Clearly, the disjunction of all the formulae $\sigma_r$ from the subcover  is $T^*$-equivalent to $\psi(x,y)$. Since each  $\sigma_r(x,y)$ is  $T^*$-equivalent to  an $L_u$-formula, so is $\psi(x,y)$. Hence, any consistent $L$-formula $\psi(x,y)$ implying $x<y$ is $T^*$-equivalent to an $L_u$-formula.
 
An arbitrary consistent binary $L$-formula $\psi(x,y)$ is equivalent to $(\psi(x,y)\land x<y)\lor \psi(x,x)\lor (\psi(x,y)\land y<x)$. The first and the third disjunct are $T^*$-equivalent to  $L_u$-formulae by the previous considerations, and the second disjunct is a unary $L$-formula, hence $L_u$ contains a name for its solution set. This completes the proof of the theorem. 
\end{proof}

Before  continuing, we note the following naturally imposed question:

\begin{que}
Is there an $L$-free condition characterizing, up to definitional equivalence, pure linear orders  among all  linearly ordered $L$-structures?
\end{que}
 
The next condition, called {\it  linear binarity}, was stated in a topological form as a property of theories of colored orders in Rubin's   Lemma 7.9 in \cite{Rubin}. It was re-formulated by  Pierre Simon in \cite{Simon2} and  \cite{Simon}.
 
\medskip\noindent
{\bf (LB)} \ For all increasing sequences $a_1<a_2<\dots <a_n$ of elements of a model of $T$:
$$\bigcup_{1\leqslant i<n} \tp_{x_i,x_{i+1}}(a_i,a_{i+1})\vdash \tp_{x_1,\dots ,x_n}(a_1,\dots ,a_n).$$

\begin{lem}\label{lem LB equiv} Each of the following conditions is equivalent with (LB):
\begin{enumerate}
\item Every  type $p(x_1,\dots ,x_n)\in S_n(T)$ that implies $x_1<\dots <x_n$  is the unique completion of the type $\bigcup_{1\leqslant i<n} p_{\strok (x_i,x_{i+1})}$, where each  $p_{\strok (x_i,x_{i+1})}$ consists of all the formulae from $p$ having no free variables other than $x_i$ and $x_{i+1}$.
\item Every  formula in free variables $x_1,\dots ,x_n$  that implies $x_1<\dots <x_n$ is $T$-equivalent to a Boolean combination of formulae of the form $\psi(x_i,x_{i+1})$ ($1\leqslant i< n$).  
\item For every  model  $M\models T$,   $a\in M$ and automorphism $f\in \Aut(M)$  fixing $a$, the mapping defined by: $g(x)=f(x)$ for $x\leqslant a$ and $g(x)=x$ for $a<x$, is an automorphism of $M$.
\end{enumerate}
\end{lem}
\begin{proof}
 (1) is a restatement of (LB) and   (1)$\Leftrightarrow$(2) is a straightforward consequence of compactness. 
 
\smallskip 
(1)$\Rightarrow$(3) Suppose that   $b_1<\dots <b_m<a<c_1<\dots <c_n$ are elements of $M$, $f\in\Aut(M)$ fixes $a$ and that $g$ is defined as in (3). Then we have:
\begin{eqnarray*}
\tp(b_1,\dots ,b_m,a,c_1,\dots ,c_n)&=& \tp(f(b_1),\dots ,f(b_m),a,c_1,\dots ,c_n)\\
&=&\tp(g(b_1),\dots ,g(b_m),g(a),g(c_1),\dots ,g(c_n));
\end{eqnarray*}
where the first equality follows by (1) because 
all the corresponding  pairs of consecutive elements of tuples  $(b_1,\dots ,b_m,a,c_1,\dots ,c_n)$ and $(f(b_1),\dots ,f(b_m),a,c_1,\dots ,c_n)$ have the same type, while the second holds by the definition of $g$.   It follows that $g$ is an automorphism. 

\smallskip
(3)$\Rightarrow$(LB) We may work in a saturated model, and we use induction on $n$. For $n=1$ there is nothing to be proved. Let $a_0<a_1<\dots <a_n$ and $a_0'<a_1'<\dots <a_n'$ be two tuples with $\tp(a_i,a_{i+1})=\tp(a_i',a_{i+1}')$ for $0\leqslant i<n$. By the induction hypothesis we have $\tp(a_1,\dots ,a_n)= \tp(a_1',\dots ,a_n')$, and we may find an automorphism $h$ mapping $(a_1',\dots ,a_n')$ to $(a_1,\dots ,a_n)$. Set $a_0''=h(a_0')$. It suffices  to prove $\tp(a_0,a_1,\dots ,a_n)= \tp(a_0'',a_1,\dots ,a_n)$. Since $\tp(a_0,a_1)= \tp(a_0'',a_1)$, there is an automorphism $f$ fixing $a_1$ and mapping $a_0''$ to $a_0$. By (3),  $f$ can be redefined as identity on $[a_1,\infty)$, so   $f$ maps $(a_0'',a_1,\dots ,a_n)$ to $(a_0,a_1,\dots ,a_n)$,  finishing the proof.
\end{proof}

Motivated by condition (3) of the previous lemma, we say that $T$ satisfies   {\it the strong linear binarity} condition if:

\medskip\noindent
{\bf (SLB)} For every  model  $M\models T$,  initial part $C\subset M$  and automorphism $f\in \Aut(M)$  fixing $C$ setwise, the mapping defined by: $g(x)=f(x)$ for $x\in C$ and $g(x)=x$ for $x\notin C$, is an automorphism of $M$.

\medskip\noindent 
We say that a linearly ordered structure satisfies the (strong) linear binarity condition if its complete theory does so.

\begin{rmk}
(a)  (LB) is a special case of (SLB) when we set $C= (-\infty,a]$ for arbitrary $a$.

\smallskip
(b)  In the definition of (SLB) one can take $C$ to be any convex set, not just an initial part. Namely, if $C$ is a  convex set fixed setwise  by $f$, we can consider the initial part $C'=\{x\mid x<C\}$ which is also fixed setwise by $f$. By (SLB) the function $h= f^{-1}\strok C'\cup id_{M\smallsetminus C'}$ is an automorphism. By (SLB) again, the function  $g$ which agrees with $g=f\circ h$ on the initial part $C'\cup C$ and is identity elsewhere, is also an automorphism. Note that $g$ agrees with $f$ on $C$ and is the identity elsewhere. 
\end{rmk}

\begin{rmk}
In (SLB) we may, equivalently, replace ``every model'' by: some $\omega_1$-saturated, strongly $\omega_1$-homogeneous model.  Indeed, suppose that an initial part $C$ of some model $M$, $f\in\Aut(M)$, $\bar b\in C$ and $\bar a\in M\smallsetminus C$ witness the failure of (SLB): $f(C)=C$ and $\tp(\bar b,\bar a)\neq\tp(f(\bar b),\bar a)$; we may  also assume that $\max  (\bar b) \leqslant\max (f(\bar b))$. Recursively define: $\bar b_0=\bar b$ and $\bar b_{n+1}= f(\bar b_n)$. If $N\models T$ is $\omega_1$-saturated, then a copy of $(\bar a,\bar b_0,\bar b_1,\dots )$ can be found in $N$; denote it by $(\bar a',\bar b_0',\bar b_1',\dots )$. 
If $N$ is in addition strongly $\omega_1$-homogeneous, then one of its automorphisms, say $g$,  moves   $(\bar b_0',\bar b_1',\dots )$ to  $(\bar b_1',\bar b_2',\dots )$. Then  $g$ fixes $C'=\bigcup_{n\in\omega} (-\infty,\max(\bar b_n'))$, so   $\bar a'$,  $\bar b_0'$ and $\bar b_1'=g(\bar b_0')$ witness the failure of (SLB) in $N$. 
\end{rmk}

\begin{lem}\label{Lem ccel orders have SLB}
Every \ccel-order satisfies  (SLB). 
\end{lem} 
\begin{proof}
Suppose that $(M,<,P_i,E_j)_{i\in I, j\in J}$ is a \ccel-order and that $f\in \Aut(M)$ fixes an initial part $C\subseteq M$.  To prove that the mapping $g$,  defined as in (SLB), is an automorphism it suffices to verify that for all  tuples $\bar a$ from $C$ and $\bar b$ from outside of $C$, the tuples $(\bar a,\bar b)$ and $(f(\bar a),\bar b)$ satisfy the same atomic formulae.  For formulae $x<y$ this is true since $f$ is an automorphism fixing $C$ setwise, and for unary predicates since $f$ is an automorphism. Consider the formula $E_j(x,y)$. If we interpret both $x$ and $y$ in $\bar a$, then the conclusion follows by $f\in\Aut (M)$;  for $x,y$  interpreted in  $\bar b$ it holds trivially. Hence, we are left with the case when the interpretations are $a_0\in\bar a$ and $b_0\in \bar b$. 
If $ (a_0,b_0)\in E_j$ holds, then   $[a_0]_{E_j}$ is the maximal element of  $(C/E_j,<)$ and so is fixed by $f$; in particular, $b_0\in [f(a_0)]_{E_j}$, i.e.\ $ (f(a_0),b_0)\in E_j$ holds. Similarly, $(f(a_0),b_0)\in E_j$    implies $ (a_0, b_0)\in E_j$, so these two are equivalent.  
\end{proof}

The linear finiteness condition for $T$ is defined by:

\medskip
\noindent
{\bf (LF)} \ For every model $M\models T$, initial part $C\subset M$ and formula $\phi(\bar x;\bar y)$, there are only finitely many $\phi$-types with parameters from $C$ that are realized in $M\smallsetminus C$. 

\begin{rmk}
By compactness it follows that an equivalent way of stating linear finiteness is: for every formula $\phi(\bar x;\bar y)$ there is an integer $n_\phi$ such that whenever $C$ is an initial part of a model $M\models T$, then there are at most $n_\phi$ \ $\phi$-types with parameters from $C$ that are satisfied in $M\smallsetminus C$. Moreover, we may restrict only to initial intervals $C=(-\infty, a)$ for $a\in M$. When stated in that form (LF) is expressible by a set of sentences, so is a part of $T$. 
\end{rmk}

Recall the well known connection between the numbers of $\phi$-types and $\phi^{op}$-types, where $\phi^{op}(\bar x;\bar y):=\phi(\bar y; \bar x)$. If there are $n$ $\phi$-types over $A$ that are realized in $B$, then there are at most $2^n$ $\phi^{op}$-types over $B$   realized in $A$: to see this, choose  representatives $\bar b_1,\dots ,\bar b_n\in B$ of all $\phi$-types over $A$ that are realized in $B$ and note that the $\phi^{op}$-type of $\bar a\in A$ over $B$ is determined by the sequence of truth values of $\phi(\bar a;\bar b_i)$ ($1\leqslant i\leqslant n$); there are at most $2^n$ such sequences. 

\begin{rmk}
(a) If we refer to  final instead of initial parts in the definition of (LF), then we get an equivalent statement.
For one direction, note that if $F$ is a final part, then $M\smallsetminus F$ is an initial part and (LF) implies that finitely many $\phi^{op}$-types over $M\smallsetminus F$ are realized in $F$, hence finitely many $(\phi^{op})^{op}$-types (i.e.\ $\phi$-types) over $F$ are realized in $M\smallsetminus F$. The other direction is proved similarly. 

\smallskip
(b)  Similarly, replacing  initial by convex parts keeps the sense of the definition.

\smallskip
(c)  Adding parameters to the language preserves (LF); in other words, if (LF) holds, then the finiteness of $\phi$-types holds also for formulae with parameters. To sketch this, assume (LF), let  $\phi(\bar x,\bar y;\bar a)$ be  a formula with parameters $\bar a$ and let $C\subseteq M$ be an initial part of a model $M$. 
Write $\bar a=\bar a'\bar a''$, where $\bar a'\in C<\bar a''$, and consider $\phi(\bar x,\bar y;\bar z)$ as $\psi(\bar x,\bar z'';\bar y,\bar z')$ where $\bar z=\bar z'\bar z''$. If $(\bar b_n\mid n\in\omega)$ were a sequence of tuples from $M\smallsetminus C$ realizing distinct $\phi$-types over $C$, then the tuples $(\bar b_n\bar a''\mid n\in\omega)$   from $M\smallsetminus C$ would realize distinct $\psi$-types over  $C$;  the latter is impossible by (LF). 
\end{rmk}

\begin{lem}\label{lem LF equivalents} Each of the following conditions is equivalent to $T$ satisfying  (LF):
\begin{enumerate}
\item For all models $M\models T$, initial parts $C\subset M$ and formulae $\phi(x;\bar y)$ ($x$ is a single variable),  only finitely many $\phi$-types with parameters from $C$ are realized in $M\smallsetminus C$.  
\item For all models $M\models T$ and initial parts $C\subset M$   at most $2^{|T|}$ types from $S_1(C)$   are realized in $M\smallsetminus C$. 
\item For all $n\in\mathbb N$,  models $M\models T$ and initial parts $C\subset M$   at most $2^{|T|}$ types from $S_n(C)$   are realized in $M\smallsetminus C$. 
\end{enumerate}
\end{lem}
\begin{proof}
(LF)$\Rightarrow$(1) is trivial. To prove (1)$\Rightarrow$(2) suppose that (1) holds,  that $C$ is an initial part of $M$ and $b\in M\smallsetminus C$.
Since  $\tp(b/C)$ is uniquely determined by $\{\tp_\phi(b/C)\mid \phi(x;\bar y)\in L\}$ there are at most $2^{|T|}$ possibilities for $\tp(b/C)$.

\smallskip
(2)$\Rightarrow$(3) Assuming (2) we prove (3) by induction on $n$.
Toward contradiction suppose that $C$ is an initial part of $M$  and that more than $2^{|T|}$ types from $S_{n+1}(C)$ are realized in $M\smallsetminus C$; let $\bar b_\alpha= b_{\alpha,1}\dots b_{\alpha,n}b_{\alpha,n+1}$ for $\alpha<(2^{|T|})^+$ be tuples from $M\smallsetminus C$ having pairwise distinct types over $C$. We may assume that for every $\alpha$, $b_{\alpha,1}<\dots<b_{\alpha,n}<b_{\alpha,n+1}$ holds. By induction hypothesis we may further assume (by extracting a subsequence) that $\tp(b_{\alpha,1},\ldots,b_{\alpha,n}/C)$ is constant. Let $M'\succ M$ be a $|M|^+$-saturated and strongly $|M|^+$-homogeneous model and let $f_\alpha\in\Aut(M'/C)$ be such that $f_\alpha(b_{\alpha,1}\dots b_{\alpha,n})= b_{0,1}\dots b_{0,n}$ for each $\alpha<(2^{|T|})^+$. Then $\tp(b_{0,1}\dots b_{0,n}f_\alpha(b_{\alpha,n+1})/C)$'s are pairwise distinct, so $\tp(f_\alpha(b_{\alpha,n+1})/Cb_{0,1}\dots b_{0,n})$'s are pairwise distinct too. From here $\tp(f_\alpha(b_{\alpha,n+1})/(-\infty,b_{0,n}])$'s are pairwise distinct, which contradicts (2) as $Cb_{0,1}\dots b_{0,n}\subseteq(-\infty,b_{0,n}]$ and $f(b_{\alpha,n+1})>b_{0,n}$.

\smallskip
(3)$\Rightarrow$(LF) Suppose that (LF) fails. Let $\phi(\bar x;\bar y)$ be a formula and let  $C$ be an initial part of a model $M$  such that  infinitely many $\phi$-types over $C$ are realized in $M\smallsetminus C$. By compactness, it is easy to find  an elementary   extension $M'$ and its initial part $C'$ containing $C$ such that $M'\smallsetminus C'$ realizes more than $2^{|T|}$ $\phi$-types over $C'$; clearly, $M'\smallsetminus C'$ realizes more than $2^{|T|}$ types from $S_{|\bar x|}(C')$.
\end{proof}
 
\begin{prop}\label{prop RB SLB LB LF}
(RB) $\Rightarrow$
(SLB) $\Rightarrow$ (LB) $\Rightarrow$(LF).
\end{prop}
\begin{proof}
(RB)$\Rightarrow$(SLB)  Assume that $M$ is saturated and that $C$, $f$, and $g$ are as in (SLB). It suffices to prove that any tuple and its $g$-image have the same type. Take a tuple $\bar a\bar b$ such that $\bar a\in C$ and $C<\bar b$; its $g$-image is the tuple $f(\bar a)\bar b$. If $\bar a=(a_1,\dots,  a_m)$ and $\bar b= (b_1,\dots,  b_n)$, where $a_1<\dots <a_m$ and $b_1<\dots <b_n$, by (RB) we need  only to check that the same finite sequences of formulae are satisfied between $a_m$ and $b_1$ as between $f(a_m)$ and $b_1$ (because this condition on other intervals trivially holds). Let $(\phi_1(x),\dots ,\phi_k(x))$ be a sequence of formulae satisfied by $c_1<\dots <c_k$ between $a_m$ and $b_1$. We may assume that $c_1,\dots ,c_l\in C$ and $C<c_{l+1}$. Then the same sequence is satisfied by $f(c_1)<\dots <f(c_l)<c_{l+1}<\dots <c_k$ between $f(a_m)$ and $b_1$. Similarly, all sequences of formulae satisfied between $f(a_m)$ and $b_1$ are satisfied between $a_m$ and $b_1$.

\smallskip
 The implication (SLB)$\Rightarrow$(LB) has already been explained, so it remains to prove 
(LB)$\Rightarrow$(LF). 
Suppose that  (LF) fails. By Lemma \ref{lem LF equivalents} there is a  model $M$ and an initial part $C\subseteq M$  such that $M\smallsetminus C$ realizes at least $(2^{|T|})^+$ complete 1-types over $C$. If $C=(-\infty,a)$ or  $C=(-\infty,a]$ for some $a\in M$, then we can find two members of $(a,+\infty)$ that have the same type over $a$  but distinct types over $(-\infty,a)$, implying the failure of  (LB). 
If $C\neq(-\infty,a)$ and $C\neq(-\infty,a]$ for every $a\in M$, then the set $\{c<x\mid c\in C\}\cup \{x<c\mid c\in M\smallsetminus C\}$ is a partial type, so there exists an elementary extension $M'$ of $M$ and $b\in M'$ realizing it. Now, in $M'$ we have $C<b<M\smallsetminus C$, so we can find two elements in   $(b,+\infty)$ that have the same type over $b$  but distinct types over $(-\infty,b)$, implying the failure of  (LB).
\end{proof}

The next examples  show  that the implications in Proposition \ref{prop RB SLB LB LF} are strict.

\begin{exm}\label{Example T_n}
Consider the structure $(\mathbb Q,<,E_n)$ consisting of the ordered rationals expanded by an equivalence relation $E_n$ having $n\geqslant 2$ classes such that each class is (topologically) dense in $\mathbb Q$. The complete theory $T_n$ of this structure is easily seen to be $\aleph_0$-categorical and to eliminate quantifiers. With these in hand, it is 
easy to count $\phi$-types and to conclude that $T_n$ satisfies (LF).

\medskip
{\it Claim 1.}  \ $T_n$ satisfies (LB) if and only if $n=2$.

That $T_2$ satisfies (LB) follows by elimination of quantifiers, so suppose  $n\geqslant 3$. Let $a<b<c$ be rational numbers from   distinct $E_n$-classes, and let $a'<b$ be such that $E_n(a',c)$ holds. Then the triples $(a,b,c)$ and $(a',b,c)$ realize distinct 3-types, while their consecutive pairs realize the same types. Hence, condition (LB) is not satisfied.

\medskip 
{\it Claim 2.} \  $T_n$ does not satisfy (SLB).  

Let  $C=(-\infty,\sqrt 2)\cap \mathbb Q$. By a standard back-and-forth construction, an automorphism $f$ of $(\mathbb Q,<,E_n)$   fixing $C$ setwise and switching two $E_n$-classes  can be constructed. Then the mapping defined by $g(x)=f(x)$ for $x\in C$ and $g(x)=x$ for $x\notin C$ is not an automorphism.  

\medskip 
Therefore, $T_2$ satisfies (LB) and $\lnot$(SLB), while $T_n$ for $n\geqslant 3$ satisfies (LF) and $\lnot$(LB).
\end{exm}

\begin{exm}  A theory satisfying   (SLB) but not (RB).\\
Consider the theory $T$ of a \ccel-order $(\mathbb Q\times \mathbb Q,<,E)$, where $<$ is the lexicographic  order and  $((q_1,r_1),(q_2,r_2))\in E$ if and only if $q_1=q_2$.  $T$ eliminates quantifiers and  has   a single complete 1-type. The formula $x<y$ has two completions in $S_2(T)$; both of them realize the same sequences of formulae, so (RB) fails. On the other hand,  (SLB) holds  by Lemma \ref{Lem ccel orders have SLB} since our structure is a \ccel-order.
\end{exm}

\begin{exm}\label{exm Lb non binary}  (LF) does not imply binarity.\\
Expand the model $(\mathbb Q,<,E_8)$   of $T_8$ by adding a $4$-ary relation $P$ in the following way: endow $\mathbb Q/E_8$ with the structure of $(\mathbb Z/2\mathbb Z)^3$ and define: $P(x,y,z,t)$ iff $[x]_{E_8}+[y]_{E_8}+[z]_{E_8}+[t]_{E_8}=0$. The theory of the expansion is not binary, although it has property (LF).
  \end{exm}

\section{Convex sets in (LF)-theories}

In this section, we deal exclusively with linearly finite structures. The main result is Theorem \ref{Thm LF descr convex sets}, in which we describe their parametrically definable convex sets. The main ingredients of the proof are contained in   Lemmas \ref{lem def init one param} and   \ref{Lema_monotone fibers} in which we prove that every parametrically definable initial part has a one-parameter  definition of the form  $x\leqslant S^n_E(a)$ or 
 $x< S^n_E(a)$ for  some convex equivalence relation $E$ and integer $n$ (these formulae are precisely defined in \ref{Definition S^n_E}). We prove in Theorem \ref{Thm_description LF functions} that every definable unary function is the union of finitely many successor-functions.

  We start by establishing a version of  monotonicity for linearly finite structures.

\begin{lem}\label{Lema_monotonicity}  (LF) Let $\phi(x,y)$ be a formula with parameters $\bar c$. Then
 $\tp(a/\bar c)=\tp(a'/\bar c)$ and $a\leqslant a'$ imply  $\sup\phi(\Mon,a)\leqslant\sup\phi(\Mon,a')$.
\end{lem}
\begin{proof}
Suppose, on the contrary,  that   $\sup\phi(\Mon,a')<\sup\phi(\Mon,a)$ holds for some elements $a<a'$ realizing the same 1-type over $\bar c$. Clearly,    $\sup\phi(\Mon,a)\neq a$ and we will continue the proof assuming   $a<\sup\phi(\Mon,a)$; the proof in the other case is similar. Then $\tp(a)=\tp(a')$ implies:  
$$a<a'<\sup\phi(\Mon,a')<\sup\phi(\Mon,a).$$
Choose a sequence $(a_n\mid n\in\omega)$   satisfying $\tp(a_n,a_{n+1}/\bar c)=\tp(a,a'/\bar c)$. Then:
$$a_0<a_1<a_2<\dots  <\sup\phi(\Mon,a_2)<\sup\phi(\Mon,a_1)< \sup\phi(\Mon,a_0).$$
Let $\psi(x;y)$ be a formula  saying $x< \sup\phi(\Mon,y)$. For each $n\in \omega$ choose an element    $b_n\in \Mon$ satisfying \ $\sup\phi(\Mon,a_{n+1})<b_n\leqslant\sup\phi(\Mon,a_n)$. Note that distinct $b_n$'s realize distinct $\psi$-types over the  parameters  $(a_n\mid n\in\omega)$. Hence,   infinitely many $\psi$-types over  the parameters  
$I=\bigcup_{n\in\omega}(-\infty,a_n)$  are realized in $\Mon\smallsetminus I$. A contradiction.
\end{proof}

\begin{lem}\label{lema o ocuvanju itp} (LF) \  Suppose that   $\Pi\in IT(A)$.
\begin{enumerate}[(a)]
\item If  $B<\Pi(\Mon)$ or $\Pi(\Mon)<B$, then \ $\Pi\ior\,\Pi\,|\, AB$.
\item If   $B<\Pi(\Mon)<C$, then \ $\Pi\ior\,\Pi\,|\, ABC$.
\end{enumerate}
\end{lem}
\begin{proof}
We will prove only that $B<\Pi(\Mon)$ implies  $\Pi\ior\,\Pi\,|\, AB$; the other parts can be proven in a similar way.  Toward contradiction, assume that $B<\Pi(\Mon)$ and that $\Pi\ior \Pi\,|\,AB$ does not hold. By Fact \ref{lema itp on cutting} there are $\bar b\in B$ and an $A\bar b$-definable initial part $D(\bar b)$ cutting $\Pi$ (we stress only parameters from $B$). Let $a_0,a_1\models\Pi$ satisfy $a_0\in D(\bar b)<a_1$ and, without loss of generality, we can assume $\tp(a_0/A)=\tp(a_1/A)$. Choose an automorphism $f\in \Aut(\Mon/A)$ satisfying $f(a_0)=a_1$ and define: $a_{n+1}=f(a_n)$,   $\bar b_0=\bar b$ and $\bar b_{n+1}=f(\bar b_n)$ (for all $n\in\omega$). Since  $\bar b_0<\Pi(\Mon)$ holds (because of $B<\Pi(\Mon)$) and since $f$ fixes $\Pi(\Mon)$ set-wise, by induction, we conclude that $\bar b_n<\Pi(\Mon)$ holds for all $n\in\omega$. Similarly, $a_n\in D(\bar b_n)<a_{n+1}$ holds for all $n\in\omega$. Therefore, all the $a_n$'s have different $\phi$-types over the initial part $\{t\in\Mon\mid t<\Pi(\Mon)\}$, where $\phi(x,\bar y)$ is the formula $x\in D(\bar y)$; this contradicts (LF).
\end{proof}

\begin{lem}\label{lema pomocna za 4 5} (LF) Suppose that an $ab$-definable initial part $D$  cuts the interval $(a,b)$ and that $D$ is neither $a$-definable nor $b$-definable. Then $D$ cuts both $\itp(b/a)$ and $\itp(a/b)$. 
\end{lem}
\begin{proof}
Since $D$ is not $a$-definable, by Fact \ref{lema itp on cutting}(c), we have that $D$ must cut some interval type $IT(a)$. 
Obviously it cuts neither those consistent with $x\leq a$, nor those greater than $\itp(b/a)$. Also, if  $\Pi\in IT(a)$  is consistent with $a<x$ and  $\Pi(\Mon)<\itp(b/a)$, then by Lemma \ref{lema o ocuvanju itp} we have $\Pi\ior\,\Pi\,|\, ab$, so $D$ does not cut $\Pi$ by Fact \ref{lema itp on cutting}(a). The only remaining possibility is that $D$ cuts $\itp(b/a)$. Similarly, $D$ cuts $\itp(a/b)$.
\end{proof}

\begin{lem}\label{lema initial part A subset D} (LF) \  Suppose that  $D$ is an $A$-definable initial part.
\begin{enumerate}[(a)]
\item If  $A\subseteq D$ and $a=\max A$, then $D$ is $a$-definable.  
\item If  $D<A$ and $a=\min A$, then $D$ is $a$-definable. 
\end{enumerate}
\end{lem}
\begin{proof}
We will prove only part (a);  the other part can be proved in a similar way. 
Assume that $D$ is $A$-definable, $A\subseteq D$ and $a=\max D$. To prove that $D$ is $a$-definable, by Fact \ref{lema itp on cutting}(c), it suffices to show that $D$ cuts no interval type  $\Pi\in IT(a)$. If $a<\Pi(\Mon)$ holds, then Lemma \ref{lema o ocuvanju itp} implies that $D$ does not cut $\Pi$. If $\Pi(\Mon)\leqslant a$ holds, then $\Pi(\Mon)\subseteq D$ and hence $D$  does not cut $\Pi$. Therefore, $D$ is $a$-definable.  
\end{proof} 

\begin{lem}\label{lem def init one param} (LF) Every $A$-definable initial part  is $A'$-definable for some $A'\subseteq A$ having at most one element.  
\end{lem}
\begin{proof} Suppose, on the contrary, that $D$  is an initial part which is not definable by a formula with at most one parameter from $A$. Without loss of generality  we can assume that $A$ is finite and non-empty. By Lemma \ref{lema initial part A subset D} we have $A\cap D\neq\emptyset$ and $A\not\subseteq D$, so   $a_0=\max(A\cap D)$ and $b_0=\min(A\smallsetminus D)$ are well-defined. Then $a_0<\sup D<b_0$ and $D$ cuts the interval $(a_0,b_0)$.  
We {\it claim} that $D$ is   $a_0b_0$-definable.   
By Fact \ref{lema itp on cutting}(a) it suffices to show that $D$ cuts no interval type from $IT(a_0b_0)$. 
 Obviously, $D$ cuts neither an interval type  satisfying $\Pi(\Mon)\leq a_0$, nor  one satisfying  $b_0\leq \Pi(\Mon)$, so we are left only with the ones  satisfying $a_0<\Pi(\Mon)<b_0$. By Lemma \ref{lema o ocuvanju itp}(b) each of them satisfies $\Pi\ior\,\Pi\,|\, A$, so  cannot be cut by $D$. This proves the claim. 
 
 \smallskip
We will say that an initial part is $ab$-good if it is $ab$-definable, cuts the interval $(a,b)$ and is neither $a$-definable nor $b$-definable. Note that  by Lemma \ref{lema pomocna za 4 5} an $ab$-good initial part cuts both $\itp(a/b)$ and $\itp(b/a)$.
By now we have proved that $D=D(a_0,b_0)$ is $a_0b_0$-good. Recursively define an increasing sequence  $(a_n)_{n\in\omega}$ and a decreasing sequence $(b_n)_{n\in\omega}$ satisfying the following properties  for all $n\in\omega$:

(1)$_n$  \  $\tp(a_n/b_{\leq n})= \tp(a_{n+1}/b_{\leq n})$ and $\tp(b_n/a_{\leq n+1})= \tp(b_{n+1}/a_{\leq n+1})$;

 (2)$_n$ \ $D(a_{n+1},b_n)$ is $a_{n+1}b_n$-good and $D(a_{n+1},b_{n+1})$ is $a_{n+1}b_{n+1}$-good;

(3)$_n$ \  $a_n\in D(a_n,b_n)<a_{n+1}<b_{n+1}\in D(a_{n+1},b_n)<b_n$.

Assume that we have $a_i$ and $b_i$ defined for $i\leqslant n$ so that $D(a_n,b_n)$ is $a_nb_n$-good and  (1)$_j$--(3)$_j$ are satisfied for all $j<n$. First, we define $a_{n+1}$. Since $D(a_n,b_n)$ is $a_nb_n$-good  it cuts $\itp(a_n/b_n)$. Then $b_n<b_{<n}$ and Lemma \ref{lema o ocuvanju itp} imply  $\itp(a_n/b_n)\vdash^i \itp(a_n/b_{\leq n})$, so $D(a_n,b_n)$ cuts $\itp(a_n/b_{\leq n})$ and there is a realization $a_{n+1}$ of $\tp(a_n/b_{\leq n})$ such that $D(a_n,b_n)<a_{n+1}$.  From  $\tp(a_n/b_n)=\tp(a_{n+1}/b_n)$  we conclude that $D(a_{n+1},b_n)$ is $a_{n+1}b_n$-good and hence cuts $\tp(b_n/a_{n+1})$. Then Lemma \ref{lema o ocuvanju itp} implies  $\itp(b_n/a_{n+1})\vdash^i\itp(b_n/ a_{\leq n+1})$, so $D(a_{n+1},b_n)$ cuts $\tp(b_n/a_{\leq n+1})$  and  
there is a realization $b_{n+1}$ of $\tp(b_n/a_{\leq n+1})$ such that $b_{n+1}\in D(a_{n+1},b_n)$.  Then $\tp(b_n/a_{n+1})=\tp(b_{n+1}/a_{n+1})$ implies  that $D(a_{n+1},b_{n+1})$ is  $a_{n+1}b_{n+1}$-good; in particular, $a_{n+1}<b_{n+1}$. Hence $D(a_{n+1},b_{n+1})$ is  $a_{n+1}b_{n+1}$-good and (1)$_n$--\,(3)$_n$ are satisfied; the recursion is well defined.  

As a corollary of our construction  we have:

(4) \  $D(a_n,b_m)$ is an initial part and $D(a_n,b_m)\subseteq D(a_n,b_n)$ for all $n< m$.

\noindent This follows by Lemma \ref{Lema_monotonicity} since  (1)$_{\geqslant n}$ implies $\tp(b_n/a_n)=\tp(b_m/a_n)$ and  (3)$_{\geqslant n}$ implies  $b_m<b_n$.  
Let  $I=\bigcup_{n\in\omega}(-\infty,a_n)$. Clearly, $I$ contains all the $a_n$'s and none of the $b_n$'s. Consider the formula $\phi(x;y,z)$ given by $z\in D(x,y)$. We claim that pairs $(b_n,b_{n+1})$ realize pairwise distinct $\phi$-types over $I$. It is enough to note that $a_{n+1}$ witnesses that the $\phi$-type of $(b_n,b_{n+1})$ is different from any $\phi$-type  of   $(b_m,b_{m+1})$ for all $m>n$. Indeed, $b_{n+1}\in D(a_{n+1},b_n)$ by (3), but $D(a_{n+1},b_m)\subseteq D(a_{n+1},b_{n+1})$ by (4), so $b_{m+1}\notin D(a_{n+1},b_m)$ because $D(a_{n+1},b_{n+1})\subseteq I$ by (3) and $b_{m+1}\notin I$. This contradiction   finishes the proof.
\end{proof}

  We will now show that the formula defining an initial part can be chosen in a specific form. 
A formula $\phi(x,y)$ is called {\em monotone} if it defines a monotone relation  on $(\Mon,<)$, i.e.\ if $(\phi(\Mon,a)\mid a\in\Mon)$ is an $\subseteq$-increasing sequence of initial parts of $\Mon$. Examples of monotone formulae are $x<f(y)$ and $x\leqslant f(y)$, where $f$ is a unary, definable and increasing function. By a {\em monotone definition} of an $a$-definable initial part $D\subseteq \Mon$ we will mean a monotone formula $\phi(x,y)$  satisfying $\phi(\Mon,a)=D$.

\begin{lem}\label{Lema_nonotone definiton}
(LF) Every $a$-definable initial  part   has a monotone  definition.
\end{lem}
\begin{proof}Suppose that   $\phi(x,a)$ defines an initial part and let $p=\tp(a)$. By Lemma \ref{Lema_monotonicity} we have:
$$p(y)\cup p(z)\cup\{y<z\}\vdash \forall  x(\phi(x,y)\rightarrow\phi(x,z)).$$ 
By compactness, there is a formula $\theta(y)\in p(y)$ satisfying: 
$$\theta(y)\wedge\theta(z)\wedge y<z\forces \forall x(\phi(x,y)\rightarrow\phi(x,z)).$$
Then the formula \   $\exists v (\theta (v)\land v\leqslant y\land \phi(x,v))$ is a  monotone definition of   $\phi(\Mon,a)$.  
\end{proof}

\begin{dfn}\label{Definition S^n_E} 
Let $E$ be a definable, convex equivalence relation   and   $N\in \mathbb Z$. 
\begin{enumerate}
\item    For $N>0$ fix a formula $S^N_E(x,y)$   expressing that the $N$-th consecutive $E$-class succeeding  the class $[x]_E$ exists and contains $y$;
\item   $S^0_E(x,y)$ is $E(x,y)$ and for $N<0$  $S^N_E(x,y):=S^{-N}_E(y,x)$; 
\item   $x< S^N_E(y):= \exists z\, S^N_E(y,z)\land \forall z\,( S^N_E(y,z)\rightarrow x< z)$;
\item $x\leq S^N_E(y):= x<S^N_E(y)\vee S^N_E(y,x)$;    
\item  $S^N_E(y)<x$ and $S^N_E(y)\leqslant x$ are defined  similarly;
\item If $E_1,E_2$ are definable, convex equivalence relations and $N_1,N_2$ integers, then we define:  $S^{N_1}_{E_1}(y)<x<S^{N_2}_{E_2}(z):=S^{N_1}_{E_1}(y)<x \land x<S^{N_2}_{E_2}(z)$; analogously, similar formulae are defined.
\end{enumerate}
\end{dfn}

For $a\in \Mon$ the fiber $S^N_E(a,\Mon)$ will be denoted by $S^N_E(a)$.   $S^N_E(a)$ is   the $N$-th consecutive $E$-class succeeding/preceding the class   $[a]_E$, if such a successor exists; otherwise $S^N_E(a)=\emptyset$. 
 Hence,  if  $S^N_E(a)=\emptyset$, then each of $x<S^N_E(a)$, $x\leqslant S^N_E(a)$, $S^N_E(a)<x$, and $S^N_E(a)\leqslant x$ is inconsistent. 

\begin{rmk}
(1)  The formulae  $x< S^{N+1}_E(a)$ and $x\leqslant S^N_E(a)$ are equivalent if and only if $S^N_E(a)=\emptyset$ or $S^{N+1}_E(a)\neq\emptyset$. 

(2) Similarly,  $x\leqslant S^N_E(a)$ and $\lnot(S^N_E(a)<x)$ are equivalent if and only if $S^{N}_E(a)\neq\emptyset$ . 

(3) The formula  $x< S^{N}_E(y)$   defines an initial part   of a model for each fixed parameter value for $y$, but it may not be monotone. The reason for that lies exclusively in  the non-existence of  $N$-th successors. Namely, there may exist elements $a<b$ such that the $N$-th $E$-class succeeding $[a]_E$ exists, but the $N$-th $E$-class succeeding $[b]_E$ does not; then $x<S^N_E(a)$ defines a proper initial part, while $x<S^N_E(b)$ is inconsistent.
\end{rmk}

\begin{lem}\label{Lema_monotone fibers}(LF)
Suppose that $D(a)$ is an $a$-definable initial part of $\Mon$.  
\begin{enumerate}[(a)]
\item  If $a\in D(a)$, then $D(a)$ is  defined by a formula of the form  $x\leqslant S^{N}_{E}(a)$ for some $N\geqslant 0$ and   definable convex equivalence $E$.
\item If $a\notin D(a)$, then $D(a)$ is  defined by a formula of the form  $x< S^{N}_{E}(a)$ for some $N\leqslant 0$ and   definable convex equivalence $E$.
\end{enumerate}
\end{lem}
\begin{proof} By Lemma \ref{Lema_nonotone definiton} there exists a monotone formula $\phi(x,y)$   such that $D(a)=\phi(\Mon,a)$. Monotonicity of  $\phi(x,y)$ implies that each of  $(\phi(\Mon,y)\mid y\in \Mon)$  and $(\lnot\phi(x,\Mon)\mid x\in \Mon)$ is an increasing sequence  of initial parts of $(\Mon,<)$. Let $E$ be the equivalence relation defined by $\phi(u,\Mon)=\phi(v,\Mon)$. Clearly, it is definable  and, by monotonicity, it is convex. 

\smallskip
(a)\  Suppose that $a\in  D(a)$.  We {\em claim} that  $D(a)$ meets only finitely many $E$-classes on  $[a,\infty)$. Otherwise, by compactness and saturation of $\Mon$, we can find an infinite increasing sequence $[a]_{E}<[b_0]_{E}<[b_1]_{E}<\dots $ such that each $[b_n]_{E}$ is contained in $D(a)$.  
By monotonicity and the definition of $E$ we have an increasing sequence of initial parts $ \lnot\phi(a,\Mon)\subset\lnot\phi(b_0,\Mon)\subset\lnot\phi(b_1,\Mon)\subset\dots $.  For each $n\in\omega$ we have $b_n\in D(a)$, so $\phi(b_n,a)$ holds, i.e.\ $a\notin \lnot\phi(b_n,\Mon)$; since   $\lnot \phi(b_n,\Mon)$ is an initial part we deduce $\lnot\phi(b_n,\Mon)<a$. So:  
$$\sup\lnot\phi(b_0,\Mon)<\sup\lnot\phi(b_1,\Mon)<\sup\lnot\phi(b_2,\Mon)<\dots <a<b_0<b_1<\dots $$
Arguing  as in  the proof of Lemma \ref{Lema_monotonicity} we deduce that this situation is impossible in theories satisfying (LF). This completes the proof of the claim. 
 
Let $[a_0]_E<[a_1]_E<\dots <[a_N]_E$,  where $a_i\in D(a)\cap[a,+\infty)$,  be the sequence of all  $E$-classes meeting $D(a)\cap[a,+\infty)$. These classes are consecutive   because $D(a)$ is a convex set. We will finish  the proof  by showing that $[a_N]_E\subseteq D(a)$ holds.   So suppose that $b\in[a_N]_E$ and $a_N<b$. Then
$a_N\in D(a)= \phi(\Mon,a)$ implies  $\models\phi(a_N,a)$. Combining with  $\phi(a_N,\Mon)=\phi(b,\Mon)$ we get $b\in\phi(\Mon,a)=D(a)$. 

\smallskip 
(b) If $D(a)=\emptyset$, then it is defined by $x<S^0_E(a)$ where $E$ is the full relation. Assuming $D(a)\neq \emptyset$,  the rest of the proof is quite similar to that of part (a). Using the same equivalence relation $E$, one proves that the complement   $\Mon\smallsetminus D(a)$ meets only finitely many $E$-classes below $a$,  with the least among them, say $N$-th below the class of $a$, being completely contained in the complement; then $D(a)$ is  defined by $x<S^{-N}_E(a)$. 
\end{proof} 

By now we have dealt only with initial parts. To describe the final parts, it suffices to note that reversing the order of a linearly ordered structure does not affect its linear finiteness. Then we deduce that any $a$-definable final part $D(a)$ has a definition of the form   $S^{-N}_E(a)\leqslant x$ or $S^N_E(a)< x$. 

\begin{thm}\label{Thm LF descr convex sets}
(LF) Every parametrically definable convex set $C$ has a definition of the form
\begin{center}
$S^{-m}_{E_1}(a)\leqslant x<S^{-n}_{E_2}(b)$,  $S^{-m}_{E_1}(a)\leqslant x\leqslant S^{n}_{E_2}(b)$,  $S^{m}_{E_1}(a)< x\leqslant S^{n}_{E_2}(b)$, or $S^{m}_{E_1}(a)< x < S^{-n}_{E_2}(b)$
\end{center}
for some  non-negative integers $m,n$ and    definable convex equivalence relations  $E_1,E_2$; if $C$ is definable with parameters from $A\neq\emptyset$, then   $a,b$ can be chosen from $A$. In particular,  $C$ is a Boolean combination of intervals and classes of definable  convex equivalence relations.
\end{thm}
\begin{proof} Suppose that $C\subseteq \Mon$ is a convex, $A$-definable set.  Then there is an $A$-definable initial part $I$ and an $A$-definable final part $F$ such that $C=I\cap F$.
By Lemma \ref{lem def init one param} $I$ is definable over a single parameter $b$ which may be chosen from $A$ provided that $A\neq\emptyset$. By Lemma \ref{Lema_monotone fibers} $I$ is defined by a formula of the form $x<S^{-n}_{E_2}(b)$ or $x\leqslant S^{n}_{E_2}(b)$.  Similarly, $F$ is $a$-definable by a formula of the form $S^{-m}_{E_1}(a)\leqslant x$ or $S^{m}_{E_1}(a)< x$. The first conclusion follows. 

For the second conclusion,  it suffices to inspect cases of  Lemma \ref{Lema_monotone fibers}.  If $b\in I$, then $I$ is the union of the interval $(-\infty,b)$ and the $n+1$ consecutive $E_2$-classes starting with $[b]_{E_2}$; if $b\notin I$, then $I$ is the difference of the interval $(-\infty,b)$ and the $m$ consecutive $E_2$-classes ending with $[b]_{E_2}$. Hence $I$ and $F$ are Boolean combinations of intervals and classes of convex equivalence relations; so is $C=I \cap F$. 
\end{proof}

The following fact, which will be used later, is an immediate corollary of the theorem.

\begin{cor}\label{Cor LFconvex a definable sets}(LF) Every $a$-definable convex set has a definition of the form
\begin{center}
$S^{-m}_{E_1}(a)\leqslant x<S^{-n}_{E_2}(a)$, \ $S^{-m}_{E_1}(a)\leqslant x\leqslant S^{n}_{E_2}(a)$  \ or \  $S^{m}_{E_1}(a)< x\leqslant S^{n}_{E_2}(a)$,
\end{center}
for some definable, convex equivalences $E_1$ and $E_2$ and non-negative integers $m,n$. In particular, every  convex $a$-definable  set $C$ not containing $a$ can be represented in the form $C_1\smallsetminus C_2$, where each $C_i$ is a union of finitely many consecutive $E_i$-classes.
\end{cor}

Although $x<S^N_E(y)$ and $x\leqslant S^N_E(y)$ are not necessarily   monotone formulae, they are nearly so: for example, $x<S^N_E(y)$ defines a monotone relation  between $(\Mon,<)$  and the solution set of  $\exists zS_E^N(y,z)$.
We now show that any monotone relation can be  defined piecewise in that way.

\begin{prop}\label{Prop_description monotone formula}
(LF) \ Every monotone formula is equivalent to a finite  disjunction of the form $\bigvee_{i\in I}(\theta_i(y)\land \psi_i(x,y))$ in which:

--  each formula $\psi_i(x,y)$ is   of the form   $x\leqslant S^{N_i}_{E_i}(y)$ or $x<S^{-N_i}_{E_i}(y)$ for some convex equivalence relation $E_i$ and  non-negative integer $N_i$; and

--  formulae $\{\theta_i(y)\mid i\in I\}$ are pairwise contradictory. 
\end{prop}
\begin{proof}Suppose that $\phi(x,y)$ is a monotone formula. For each type $p(y)\in S_1(T)$ consistent  with $\exists x\phi(x,y)$, by Lemma \ref{Lema_monotone fibers},  there is a formula $\psi_p(x,y)$   of the required form such that \ $p(y)\vdash \forall x\,(\phi(x,y)\leftrightarrow\psi_p(x,y))$. By compactness, there is $\theta_p(y)\in p(y)$ implying $\exists x\phi(x,y)$ and: 
\begin{equation}
 \models \forall y(\theta_p(y)\rightarrow \forall x\,(\phi(x,y)\leftrightarrow\psi_p(x,y)))\,.
\end{equation}
The sets  $ [\theta_p]$ for $p\in [\exists x\phi(x,y)]$ form  a cover of $[\exists x\phi(x,y)]\subseteq S_1(T)$. 
By compactness, there is a finite subcover  $\{[\theta_{p_i}]\mid i\leqslant n\}$. 
Replace each $\theta_{p_i}(y)$ by $\theta_{p_i}(y)\land\bigwedge_{j<i}\lnot\theta_{p_j}(y)$ and note that (1) holds with $p$ replaced by any $p_i$. Hence \ $\models \forall xy(\phi(x,y)\leftrightarrow \bigvee_{i\leqslant n}(\theta_{p_i}(y)\land \psi_{p_i}(x,y)))$. 
\end{proof}

Pierre Simon in \cite{Simon} proved that the complete theory of a colored order expanded by naming all definable unary predicates and all definable monotone relations eliminates quantifiers. Since colored orders are linearly finite structures, Proposition \ref{Prop_description monotone formula} applies and   provides a description of definable monotone relations. Therefore, if we expand a colored order by naming all definable relations $x<S^n_E(y)$ and $x\leqslant S^n_E(y)$, where $n$ is an integer and $E(x,y)$ defines a convex equivalence relation, then the theory of the expanded structure eliminates quantifiers.

\begin{thm}\label{Thm_description LF functions}
Suppose that the function $f:M\to M$ is definable in a linearly finite structure $(M,<...)$. Then $f(x)=y$ can be defined by a finite  disjunction of the form $\bigvee_{i\in I}(\theta_i(x)\land S^{N_i}_{E_i}(x)=\{y\})$ 
for some convex equivalence relations $E_i$ and  integers $N_i$  ($i\in I$). In particular,   $f$ is the union of finitely many definable, increasing functions. 
\end{thm}
\begin{proof}
Without loss of generality, we may assume $M=\Mon$. Let $a\in \Mon$ and $p=\tp(a)$. We will find a formula $\theta_p(x)\in p(x)$ by distinguishing two cases.

\smallskip
Case 1. \ $a\leqslant f(a)$. \  In this case Lemma \ref{Lema_monotone fibers}(a) applies to the initial interval $(-\infty,f(a)]$:
it is defined by $z\leqslant S^N_E(a)$ for some non-negative integer $N$ and a definable, convex equivalence relation $E$. This implies that   $f(a)$ is the maximal element of the class $S^N_E(a)$. Consider the convex partition  $\Mon/E=\{[b]_E\mid b\in \Mon\}$ of $\Mon$ and 
let $D$ be the set of maximums of all $E$-classes with maximal elements. 
Then $\{[b]_E\smallsetminus D\mid b\in \Mon\}\cup \{\{d\}\mid d\in D\}$  is a convex partition of $\Mon$; let
 $E_p$ be the corresponding equivalence relation. Clearly, $E_p$ is a definable, convex equivalence relation splitting each $E$-class into at most two classes, so $\{f(a)\}= S^{N_p}_{E_p}(a)$ holds for some non-negative integer $N_p$. Choose a formula  $\theta_p(x)\in p$  expressing: $\{f(x)\}= S^{N_p}_{E_p}(x)$. 
 
Case 2. \ $f(a)<a$. \ In this case Lemma \ref{Lema_monotone fibers}(b) applies to the initial interval $(-\infty, f(a))$: it is defined by  $z< S^{N}_E(a)$ for suitable chosen $E$ and $N\leqslant 0$. Notice that $f(a)$ is the minimal element of the class $S^{N}_E(a)$. As in the previous case, we find $E_p$ slightly refining $E$ so that $\{f(a)\}$ is a single $E_p$-class and $\{f(a)\}=S^{N_p}_{E_p}(a)$ holds for some negative integer $N_p$.

\smallskip
By compactness,  $f(x)=y$ is defined by $\bigvee_{i\in I}(\theta_i(x)\land S^{N_i}_{E_i}(x)=\{y\})$  for some finite $I\subseteq S_1(T)$; we may slightly modify $\theta_i$'s so that they are pairwise contradictory. Finally, if we denote $f_i=f\strok\theta_i(\Mon)$, then  $f_i$ is increasing and $f=\bigcup_{i\in I}f_i$. 
\end{proof}

\section{Almost convex equivalence relations  and (LB)}

In this section, we will describe one-parameter definable subsets of linearly finite structures and all parametrically definable subsets of (LB)-structures.  

\begin{dfn}
An equivalence relation $R$ on   a linearly ordered set is  {\it almost convex}     if there is a convex equivalence relation $E$ coarser than $R$ such that  $R$ splits each $E$-class into finitely many classes. 
\end{dfn} 

The above mentioned description will be in  terms of classes of definable, almost convex equivalence relations. Note that these include all  unary definable sets  since $\phi(\Mon)$ is a class of the   relation defined by $\phi(x)\leftrightarrow \phi(y)$. This relation is an almost convex equivalence relation as it splits $\Mon$ (the unique class of the trivial equivalence relation) into at most two classes.

\begin{rmk}
Let $R$ be an  equivalence relation   on a linearly ordered set. Among the convex equivalences coarsening $R$ there  exists the finest one: If $X^{conv}$ denotes  the convex closure of $X$ and   $r(X)=\{y\mid \exists x(x\in X\land R(x,y))\}$, then there is a minimal superset of $X$, denoted by $\cl(X)$, which is closed under operations $^{conv}$ and $r$. It is straightforward to verify that the set $\{\cl([x]_R)\mid x\in \Mon\}$ is a convex partition of $\Mon$, and that the induced equivalence relation $E$ is the finest convex equivalence relation coarsening $R$; $E$ is the {\it  convex closure of $R$}.
If $R$ is a definable almost convex equivalence relation on $\Mon$, then its convex  closure $E$ is definable, too. To see this, first note that for each $a\in\Mon$ the set $\cl([a]_R)$ is obtained by applying operation  $(r\circ ^{conv})^{n_a}$  to $[a]_R$, where $n_a$ is the number of $R$-classes contained in $[a]_E$. By compactness and saturation the $n_a$'s are uniformly bounded, so $E$ is definable.  
\end{rmk}

The main technical result of this section is the  following proposition.

\begin{prop}\label{prop LF one param}(LF) Every $a$-definable set  is a finite Boolean combination  of the interval  $(a,\infty)$  and  classes of  definable, almost convex equivalence relations.
\end{prop}

In the proof, we will need some extra notation and a few lemmas.   We say that a convex set $C$ is {\it $D$-good} if there is a set $D'$ which is the union of finitely many classes of some definable, almost convex equivalence relation such that $C\cap D=C\cap D'$; i.e. $D$ and $D'$ agree on $C$. 
We will prove the proposition by showing that for a fixed $D=\phi(\Mon,a)$ there exists a convex $a$-definable partition $\mathcal P$ of $\Mon$ such that  each $C\in\mathcal P$ is $D$-good and $\{a\}\in \mathcal P$. By Corollary \ref{Cor LFconvex a definable sets} each $C\in\mathcal P$ is a Boolean combination of classes of definable convex equivalence relations, so $D\cap C$ is a Boolean combination of classes of  definable, almost convex equivalence relations; hence so is $D=\bigcup_{C\in\mathcal P}(C\cap D)$.  
In fact, each $C\in\mathcal P$ in such a decomposition will be chosen to be either the whole class or an appropriate end-part of a class of a definable convex equivalence relation, so   in order to prove that $C_i$ is $D$-good we will distinguish these two cases. 

For a partitioned formula $\phi(x;\bar z)$ and   a  convex set $C$ denote by   $x\equiv _\phi y\,(C)$  the equivalence relation ``$x$ and $y$ have the same $\phi$-type over $C$''; if $C$ is parametrically  definable,   then we will use the same notation for  the defining formula. Under the (LF) assumption, only finitely many classes of this relation intersect $\Mon\smallsetminus C$ nontrivially.   
   
\begin{lem}\label{Lema_notinC}
(LF) If $C$ is a class of a definable convex equivalence relation $E$ and $a\notin C$, then  $C$ is $\phi(\Mon,a)$-good for every $\phi(x,y)$.
\end{lem}
\begin{proof} We will prove the claim assuming $a<C$; the proof in the other case is similar.
Consider the relation $R$ defined by: $E(y,z)\ \land\ y\equiv_\phi z\, ([y]_E^-)$, where $[y]_E^-:=\{u\mid u<[y]_E\}$. Note that $R$ is   a definable equivalence relation refining $E$. For each $E$-class $C'$ there are finitely many $\phi$-types over the parameters below $C'$ that are realized in $C'$, so $R$ splits each $E$-class into finitely many classes and $R$ is almost convex. Let $C_0,C_1,\dots ,C_n$ be the list of all $R$-classes contained in $C$. All the elements of $C_i$ have the same $\phi$-type over $a$ because  $a<C$, so  $\phi(\Mon,a)$ either contains  or is disjoint from $C_i$. Clearly, $\phi(\Mon,a)\cap C$ is a union of  finitely many $C_i$'s, and  $C$ is $\phi(\Mon,a)$-good. 
\end{proof}

\begin{lem}\label{lem finitely positions on a class} (LF) Suppose that $E$ is a definable convex equivalence relation and $\phi(x,y)$ is a formula such that $\phi(x,y)\forces E(x,y)$. If $C$ is an $E$-class such that the set $\{\phi(\Mon,a)\mid a\in C\}$ is finite, then  $C$ is $\phi(\Mon,a)$-good for all $a\in C$.
\end{lem}
\begin{proof}
Let $n=|\{\phi(\Mon,a)\mid a\in C\}|$. Denote by $\psi(x)$ a formula saying ``the cardinality of $\{\phi(\Mon,a)\mid a\in [x]_E\}$ is $n$''. Clearly, $\psi(\Mon)$ is an $E$-closed set containing $C$. Define:  
$$R(x,y):= E(x,y)\ \wedge\ (\ \psi(x)\ \rightarrow \ x\equiv_\phi y\ ([x]_E)\ ).$$
Clearly, $R$ defines an equivalence relation refining $E$. Moreover, each $E$-class   contained in $\psi(\Mon)$ is divided in at most $2^n$ $R$-classes: if $\phi(\Mon,a_0),\dots ,\phi(\Mon,a_{n-1})$ are all different members of $\{\phi(\Mon,y)\mid y\in[a_0]_E\}$, then $R$-classes on $[a_0]_E$ are exactly non-empty sets among $[a_0]_E\cap \bigcap_{i<n}\phi^{f(i)}(\Mon, a_i)$ for $f\in 2^n$. Therefore, $R$ is almost convex. For any  $a\in C$ we have  $\phi(\Mon,a)\subseteq [a]_E$, so $\phi(\Mon,a)$ is   a union of finitely many $R$-classes and hence $C$ is $\phi(\Mon,a)$-good. 
\end{proof}

For a formula $\phi(x;y)$ and $a\leq b$, define $N_\phi(a,b)$ to be the number of $\phi$-types with parameters in $(-\infty,a]$ that are realized in $[b,+\infty)$; by (LF) assumption, $N_\phi(a,b)$ is a finite number. The following properties are easily verified:

\begin{enumerate} 
\item The function $N_\phi(-,b):(-\infty,b]\to\mathbb N$ is increasing for any $b\in\Mon$, while $N_\phi(a,-):[a,+\infty)\to\mathbb N$ is decreasing for any $a\in\Mon$.
\item For a fixed $n\in\mathbb N$, the property  $N_\phi(a,x)=n$ is expressible by a formula (with parameter $a$) defining a convex set.
\end{enumerate}
For $b<a$, we leave $N_\phi(a,b)$ undefined.

\begin{lem}\label{lem convex with constant n}(LF) Suppose that $E$ is a definable convex equivalence relation  and $D(a)$ is an $a$-definable initial part of $\Mon$ such that $C_0=[a]_E\smallsetminus D(a)\neq\emptyset$ and $a<C_0$. If a formula $\phi(x,y)$ is such that $N_\phi(a,-)$ is   constant on $C_0$, then  $C_0$ is $\phi(\Mon,a)$-good.
\end{lem}
\begin{proof}
Let $n$ be the value of $N_\phi(a,-)$ on $C_0$. By Lemma \ref{Lema_nonotone definiton} we may assume that $x\in D(y)$ is  a monotone definition of $D(a)$, i.e.\ that $b\leq c$ implies $D(b)\subseteq D(c)$; also we may assume $\models \forall y \,( y\in D(y))$. Let   $\theta(y)\in \tp(a)$  be a formula saying:
\[
   [y]_E\smallsetminus D(y)\neq\emptyset\ \mbox{ \ and \ }\ ``N_\phi(y,-)\mbox{ has value }n\mbox{ on }[y]_E\smallsetminus D(y)\mbox{''}. 
\]

\noindent{\bf Claim.} If $C$ is an $E$-class, $v\in C$ satisfies $\theta(y)$ and $u\leq v$, then the following are equivalent:
\begin{enumerate}[(1)]
\item $u\equiv_{\phi^{op}} v\ (\Mon\smallsetminus D(v))$;
\item $u\equiv_{\phi^{op}} v\ ([c,+\infty))$, for all (some) $c\in C\smallsetminus D(v)$.
\end{enumerate}

\noindent{\em Proof of the claim.} (1) obviously implies the ``all'' version of (2), and the ``all'' version of (2) obviously implies the ``some'' version of (2). To prove that the ``some'' version of (2) implies (1), assume that $c\in C\smallsetminus D(v)$ is such that $u\equiv_{\phi^{op}}v\ ([c,+\infty))$. For (1) it suffices to check that $\phi(d,u)\leftrightarrow\phi(d,v)$ holds for any $d\in (-\infty,c)\smallsetminus D(v)$, 
so   fix such a $d$ and note that   $d\in C\smallsetminus D(v)$ since $v,c\in C$ and $v<d<c$.
Since $\models\theta(v)$, by the choices of $\theta(y)$, $c$ and $d$ we get $N_\phi(v,c)=N_\phi(v,d)=n$. Choose representatives $c_1,\dots ,c_n\in[c,+\infty)$ of realizations of all  $\phi$-types over $(-\infty,v]$ realized in $[c,+\infty)$. Since $d<c$ and $N_\phi(v,d)=n$, the $c_i$'s are distinct representatives of all $\phi$-types over $(-\infty,v]$ realized in $[d,+\infty)$ as well. 
Thus, $d$ has the same $\phi$-type over $(-\infty,v]$ as some $c_i$ and $d\equiv_\phi c_i\ (\{u,v\})$. On the other hand, $u\equiv_{\phi^{op}}v\ ([c,+\infty))$, in particular $u\equiv_{\phi^{op}}v\ (c_i)$. Therefore, $\models\phi(d,u)\leftrightarrow\phi(d,v)$.
\hfill$\qed_{Claim}$

\smallskip
Consider the following refinement of  $E(y,z)$:
$$R(y,z):=E(y,z)\ \wedge\ [\ (\neg\theta(y)\wedge\neg\theta(z)) \ \vee\ (\ \theta(y)\wedge\theta(z)\wedge y\equiv_{\phi^{op}}z\ (D(\max\{y,z\})^c)\ )\ ].$$ 
By the previous claim one   easily sees that $R$ is an equivalence relation. Moreover, each $E$-class $C$ is divided into finitely many $R$-classes: $C\cap\neg\theta(\Mon)$ is a single $R$-class and $C\cap \theta(\Mon)$ is divided in at most $2^n$ many $R$-classes. For the latter, take $a_0<a_1<\dots <a_{2^n}$ in $C\cap \theta(\Mon)$. Since $N_\phi(a_{2^n},-)$ has value $n$ on $C\smallsetminus D(a_{2^n})$, there are \ $n$ \ $\phi$-types over $(-\infty,a_{2^n}]$ that are realized in $D(a_{2^n})^c$, so  at most $2^n$ $\phi^{op}$-types over $D(a_{2^n})^c$ are realized in $(-\infty,a_{2^n}]$ and for some $i<j\leq 2^n$ we have $a_i\equiv_{\phi^{op}}a_j\ (D(a_{2^n})^c)$. By the claim and monotonicity $a_i\equiv_{\phi^{op}}a_j\ (D(a_{j})^c)$, i.e.\ $R(a_i,a_j)$ holds.
Consider now the following formula:
$$\psi(x,y):= E(x,y)\wedge \theta(y)\wedge \exists z(R(y,z) \wedge D(z)<x\wedge\phi(x,z)).$$
For $R(y,z)$, we have $\psi(\Mon,y)=\psi(\Mon,z)$, so the set $\{\psi(\Mon,y)\mid y\in [a]_E\}$ is finite. Since $\psi(x,y)\forces E(x,y)$ Lemma \ref{lem finitely positions on a class} applies and  $[a]_E$ is  $\psi(\Mon,a)$-good. Thus if we prove that $\phi(\Mon,a)$ and $\psi(\Mon,a)$ agree on $C_0=[a]_E\smallsetminus D(a)$, we are done. If  $b\in [a]_E\smallsetminus D(a)$ and $\models\phi(b,a)$, then $D(a)<b$ and  by taking $z=a$ to witness the existential quantifier  we get $\models\psi(b,a)$. For the other implication, if $\models\psi(b,a)$, take $a'$ such that $R(a,a')$, $D(a')<b$ and $\models \phi(b,a')$. From $R(a,a')$ and $D(a),D(a')<b$ we have $a\equiv_{\phi^{op}}a'\ (b)$, so $\models\phi(b,a')$ implies $\models\phi(b,a)$. The proof is finished.
\end{proof}

\begin{proof}[Proof of Proposition \ref{prop LF one param}] Assume (LF) and let $D=\phi(\Mon,a)$. We will prove that $D\cap (a,+\infty)$ is a Boolean combination of classes of definable almost convex equivalence relations. By duality  the same holds for $D\cap (-\infty, a)$, so $D$ has the desired representation. 
Assume from now on that  $D\subseteq (a,\infty)$ and  let $N_\phi(a,a)=n_0$. We will define an $a$-definable convex partition $\mathcal P$ of the interval $(a,+\infty)$ such that  each $C\in \mathcal P$
 satisfies exactly one of the following two conditions:
\begin{enumerate}[(1)] 
\item  $C$ is an $E$-class of some definable convex equivalence relation $E$ and $[a]_E<C$;
\item $C$ is an end part of $[a]_E$ for some definable convex equivalence relation $E$ and $N_{\phi}(a,-)$ has constant value on $C$. 
\end{enumerate}
Define  $\mathcal V_0=\{N_\phi(a,x)\mid a<x\}$,  let $\min \mathcal V_0=j_0\leqslant n_0$ and let $S_0$ be the set defined by  
  $N_\phi(a,x)=j_0$.  
The function $N_{\phi}(a,-)$ decreases on $U_0=(a,+\infty)$, so $S_0$ is a final part of $\Mon$. Now we find an $a$-definable final part $S_0^*\supseteq S_0$ and its convex $a$-definable partition $\mathcal P_0$ whose  members each satisfy exactly one of the conditions (1) and (2).
By Corollary \ref{Cor LFconvex a definable sets} there are definable convex equivalence relations $E_{1}$ and $E_{2}$ such that $S_0=S_{0,1}\smallsetminus S_{0,2}$ and each $S_{0,i}$ is a union of finitely many consecutive $E_{i}$-classes beginning with $[a]_{E_{i}}$. 
Thus $S_0$ meets only finitely many $E_{1}$-classes and they are consecutive and $a$-definable; let them be   
$[b_0]_{E_{1}}$, $S^1_{E_{1}}(b_0)$,...., $S^{k_0}_{E_{1}}(b_0)$. Clearly, $S_0$ contains each $S^j_{E_{1}}(b_0)$ ($1\leqslant j\leqslant k_0$), while $S_0\cap[b_0]_{E_{1}}$ is an end part of $[b_0]_{E_{1}}$. We have two cases:

Case  1. \ $[b_0]_{E_{1}}\neq [a]_{E_{1}}$. \ Let  $\mathcal P_0=\{[b_0]_{E_{1}}, S^1_{E_{1}}(b_0),...., S^{k_0}_{E_{1}}(b_0)\}$ and let $S_0^*=\bigcup \mathcal P_0$. All the members of $\mathcal P_0$ satisfy  condition (1). 

Case  2. \ $[b_0]_{E_{1}}=[a]_{E_{1}}$. \ Let $S_0^*=S_0$ and let  $\mathcal P_0=\{[b_0]_{E_{1}} \cap S_0, S^1_{E_{1}}(b_0),...., S^{k_0}_{E_{1}}(b_0)\}$. Here  $[b_0]_{E_{1}} \cap S_0$ satisfies condition (2), while all the other members of $\mathcal P_0$ satisfy (1). 

\smallskip
If $S_0^*=(a,+\infty)$,  then $\mathcal P=\mathcal P_0$  is the desired partition. Otherwise, we repeat the procedure with  $U_1=(a,+\infty)\smallsetminus S_0^*$ in place of $U_0=(a,+\infty)$. 
Let $\mathcal V_1=\{N_\phi(a,x)\mid a<x \land x\in U_1\}$ and let $\min \mathcal V_1=j_1\leqslant n_0$.
By the construction we have   $\mathcal V_1\subseteq \mathcal V_0\smallsetminus \{j_0\}$ and $j_0<j_1$. Let  $S_1$ be the set defined by $x\in U_1\land N_\phi(a,x)=j_1$; then $S_1$ is a final part of $U_1$ and arguing as before, we find a final part $S_1^*\supseteq S_1$ and its convex partition $\mathcal P_1$. Continue in this way as long as it is possible. Since the $j_i$'s increase and are $\leqslant n_0$  this has to stop after finitely many steps. Then $\mathcal P=\bigcup \mathcal P_i$ is a convex partition of $(a,+\infty)$ and each $C\in \mathcal P$ satisfies either (1) or (2).
If $C$ satisfies (2), then it is $D$-good by  Lemma \ref{lem convex with constant n}; if $C$ satisfies (1), then it is $D$-good by Lemma \ref{Lema_notinC}.

We have a convex, $a$-definable partition $\mathcal P$ of the interval $(a,+\infty)$ whose members are $D$-good. 
 By Corollary \ref{Cor LFconvex a definable sets} each $C\in\mathcal P$ is a Boolean combination of classes of definable convex equivalence relations, so $D\cap C$ is a Boolean combination of classes of  definable, almost convex equivalence relations; the same holds for $D=D\cap(a,+\infty)=\bigcup_{C\in\mathcal P}(C\cap D)$.   
\end{proof}

If we in addition assume that the theory is binary, we directly derive the following description of definable sets of singletons.

\begin{thm}\label{Thm_LB description of def sets}
(LB) Every parametrically definable  set of singletons is a Boolean combination of intervals and classes of definable almost convex equivalence relations. 
\end{thm}

\begin{rmk} One may try to prove that (LB) implies that every formula $\phi(x_1,\dots ,x_n)$ is equivalent  modulo $T$ to a Boolean combination of formulae of the form $x_i=x_j$, $x_i<x_j$, unary formulae,  formulae defining almost convex equivalence relations and their $S_E^n$-like variants. However, that is not  possible. Take the structure $(\mathbb Q,<,E_3)$ from Example \ref{Example T_n}, fix  a cyclic permutation $p$ of $\mathbb Q/E_3$   and expand the structure  by adding a  binary relation defined by: $P(x,y)$ if and only if $p([x]_{E_3})=[y]_{E_3}$. The theory of the expansion has property (LB), but the formula $P(x,y)$ is not equivalent to a formula of the above form. 
\end{rmk}

Our general feeling is that a structure  satisfying (LF) can be produced, up to definitional equivalence, from a \ccel-order  by   refining   some of the convex equivalence relations into finitely many uniform (in some sense) pieces and then adding a structure to the finite quotient; this is  essentially what we did in Example \ref{exm Lb non binary} and in the previous remark. However, we could not find  an explicit quantifier-elimination type result  even for theories satisfying (LB).

\section{Strong linear binarity}

In this section, we prove Theorems \ref{teorema1} and \ref{teorema2}. They are rather corollaries of the previous results once we prove the following lemma.

\begin{lem}\label{Lema SLB ac}
(SLB) Every formula   defining an almost convex equivalence relation $R$ is equivalent modulo $T$ to an $L$-formula of the form \  
$\bigvee_{i\leqslant n}\, (\phi_i(x)\land E(x,y)\land\phi_i(y))$ 
\ 
where $E(x,y)$ defines the convex closure of $R$ and $\{\phi_i(\Mon)\mid  i\leqslant n\}$ is a definable partition of $\Mon$.
\end{lem}
\begin{proof}
Suppose that $R$ is a definable  almost convex equivalence relation on $\Mon$ and let $E$ be its convex closure. Let $a\in \Mon$ and  $\tp(a)=p$. 

\smallskip

\noindent{\bf Claim.} \ $\models R(a,y)\leftrightarrow (E(a,y)\land (\exists z)( R(z,y)\land\bigwedge_{\psi\in p}\,\psi(z)))$.

\smallskip
\noindent{\em Proof of the claim.} The left-to-right direction is clear. We prove the other one in the contrapositive. So 
suppose that for some $a'$ realizing $p$ and $b\in\Mon$ we have $\models E(a,b)\land R(a',b)\land \lnot R(a,b)$. Clearly, $a,a'$ and $b$ belong to the same $E$-class which we will  
denote   by $C$. Also $a$ and $a'$ are from distinct $R$-classes and, without loss of generality, we will assume that $a<a'$ holds. Let $f$ be an automorphism of $\Mon$ mapping $a$ to $a'$. Choose an increasing  sequence  
  $A=(a_n\mid n\in\omega)$ satisfying $a_0=a$ and $f(a_n)=a_{n+1}$ ($n\in\omega$). Then  
$\tp(a_n,a_{n+1})=\tp(a,a')$ holds for all $n\in\omega$, so all the elements of the sequence belong to $C$, with $a_n$ and $a_{n+1}$ being in distinct $R$-classes. Note also that the set $I=\bigcup_{n\in\omega}\,(-\infty, a_n]$   is an initial part of $\Mon$ and that it is fixed by $f$ (setwise). By applying condition (SLB) we may assume that $f$ is the identity on $\Mon\smallsetminus I$.  
Since the class $C$ consists of finitely many $R$-classes, at least one of them (say $C_0$) contains infinitely many members of $A$. Consider the following set of formulae \   $\Sigma(x)=\{x\in C_0\}\cup\{a_n<x\mid n\in\omega\}$. \ Every finite subset of $\Sigma(x)$  is satisfied by 
all large enough elements of $A\cap C_0$  so, by saturation, there exists an element $c\in C_0$ realizing $\Sigma(x)$. In particular, we have  $I<c$ so, by our assumption on $f$,   $f(c)=c$. Choose $a_m\in C_0$. Then  $f(a_m,c)=(a_{m+1},c)$, so $\models R(a_m,c)\leftrightarrow R(a_{m+1},c)$ and $a_m,c\in C_0$ imply $\models R(a_m,a_{m+1})$. This is impossible because  $\tp(a_m,a_{m+1})=\tp(a,a')$ implies that $a_m$ and $a_{m+1}$ are in distinct $R$-classes. \hfill$\qed_{Claim}$

\smallskip
The rest of the proof is routine. First  note that the right hand side of the equivalence in the claim is an infinite conjunction, while the left one is a single formula. By compactness, finitely many conjuncts are needed for the equivalence to hold. In fact, since $p$ is a complete type  one of them would suffice, so we can choose  $\theta_p(z)\in p(z)$ satisfying: 
\begin{center}
$\models R(a,y)\leftrightarrow (E(a,y)\land (\exists z)( R(z,y)\land\ \theta_p(z)))$.
\end{center} Hence \  $p(x)\vdash R(x,y)\leftrightarrow (E(x,y)\land \phi_p(y))$, \ where \ $\phi_p(y):= (\exists z)( R(z,y)\land\ \theta_p(z))$. By compactness again, there exists $\psi_p(x)\in p(x)$ such that 
$\psi_p(x)\vdash R(x,y)\leftrightarrow (E(x,y)\land \phi_p(y))$. Without loss of generality  assume $\models\phi_p(x)\rightarrow \exists z (E(z,x)\land \psi_p(z))$. Then   $\phi_p(x)\vdash R(x,y)\leftrightarrow (E(x,y)\land \phi_p(y))$ and,
 by extracting a finite subcover of $S_1(T)$  from   $\{[\phi_p]\mid p\in S_1(T)\}$,  we get 
$$\models R(x,y)\leftrightarrow\bigvee_{p\in S}\, (\phi_{p}(x)\land E(x,y)\land\phi_{p}(y)).$$
It remains to note  that we may modify the $\phi_{p}$'s  so that they are  pairwise contradictory.
\end{proof}

\begin{cor}\label{Cor slb2} 
(SLB) Every parametrically definable   set $D\subseteq \Mon$ is  a Boolean combination of unary definable sets, intervals and classes of  definable convex equivalence relations. In particular, $D$ has finitely many convex components on the locus of a fixed type $q\in S_1(T)$.  
\end{cor}
\begin{proof}
Suppose that $D\subseteq\Mon$ is parametrically definable.  By Theorem \ref{Thm_LB description of def sets}
$D$ is a Boolean combination of intervals and classes of definable almost convex equivalence relations;  by Lemma \ref{Lema SLB ac}  each of these classes is  a Boolean combination of unary definable sets and classes of definable convex equivalence relations. This proves the first part. To prove the second, assume that  $D$ is a Boolean combination of sets $D_1,\dots , D_n$ such that each $D_i$ is either  a unary $L$-definable set,  an interval, or a class of a convex $L$-definable equivalence. Note that in each case $D_i\cap q(\Mon)$ is a convex subset of $(q(\Mon),<)$, so that the Boolean combination has finitely many convex components.  
\end{proof} 
 
\begin{rmk}
Recall that a complete theory of linearly ordered structures  $T$ is  called  weakly quasi-o-minimal if every parametrically definable subset of a model of $T$ is a Boolean combination of unary  $L$-definable sets and convex sets. A consequence of the previous corollary
is that (SLB) implies weak quasi-o-minimality of the theory. 
\end{rmk}

 We say that a definable set $D\subseteq \Mon$ is {\it definably convex} if there is an $L$-formula $\phi(x)$ such that $D$ is a convex subset of $(\phi(\Mon),<)$. In other words, $D$ is the intersection of a unary definable set and a convex set.

\begin{prop}\label{prop slb3} (SLB)  Every parametrically definable subset of $\Mon$ can be partitioned into finitely many  definably convex  pieces (definable over the same parameter set).
\end{prop}
\begin{proof} We will prove that any $\bar a$-definable set can be represented as a union of finitely many $\bar a$-definable, definably convex sets. Having such a representation, it is not hard to produce one in which the definably convex sets are pairwise disjoint.  
Let $D=\phi(\Mon,\bar a)$  and let $q\in S_1(T)$. By Corollary  \ref{Cor slb2} the set $D\cap q(\Mon)$ has finitely many, say $<n_q$, convex components. Hence, the following set of formulae is inconsistent:
$$\bigcup_{i\leqslant n_q} q(x_i)\cup \left\{\bigwedge_{i<n_q} x_i<x_{i+1},\bigwedge_{i< n_q}\lnot(\phi(x_i,\bar a)\leftrightarrow\phi(x_{i+1},\bar a))\right\}.$$
By compactness,  the formula \ $\bigwedge_{i\leqslant n_q} \psi_q(x_i)\land\bigwedge_{i< n_q} (x_i<x_{i+1}\land \lnot(\phi(x_i,\bar a)\leftrightarrow\phi(x_{i+1},\bar a)))$ is inconsistent  for some $\psi_q(x)\in q(x)$. This means that the set $D\cap \psi_q(\Mon)$ has $<n_q$ convex components on $(\psi_q(\Mon),<)$. Clearly, each  of these components is     $\bar a$-definable and definably convex, so  $D\cap \psi_q(\Mon)$ is the union of finitely many definably convex, $\bar a$-definable sets;  note that this still holds after replacing $\psi_q(y)$ by a formula implying it.
The rest of the proof is a routine application of compactness: the union  $\bigcup_{q\in S_1(T)}[\psi_q]$  covers the space $S_1(T)$, so we can choose a finite subcover  $[\psi_{q_i}]$.  Then $D=\bigcup_{i\leqslant N}(D\cap \psi_{q_i}(\Mon))$ is a representation of $D$ as the finite union of  definably  convex, $\bar a$-definable sets. 
\end{proof}

Now we are ready to prove Theorem \ref{teorema2}. Recall that a {\it $u$-convex} formula is either a unary $L$-formula or the formula $\theta(x,y)$ which is a conjunction   of a unary $L$-formula $\psi(x)$ and one of 
\begin{center}
$S^{-m}_{E_1}(y)\leqslant x<S^{-n}_{E_2}(y)$, \ $S^{-m}_{E_1}(y)\leqslant x\leqslant S^{n}_{E_2}(y)$,  \ and \  $S^{m}_{E_1}(y)< x\leqslant S^{n}_{E_2}(y)$.
\end{center}
 By Corollary \ref{Cor LFconvex a definable sets} every $a$-definable, definably convex subset is defined by  $\theta(x,a)$ for some $u$-convex formula $\theta(x,y)$.

\begin{proof}[{\it Proof of Theorem \ref{teorema2}}] Assuming that $T$ satisfies (SLB), we will prove that every $L$-formula is  equivalent modulo $T$ to a Boolean combination of $u$-convex formulae.  Since $T$ is binary, it suffices to prove it for formulae in two free variables. Fix $\phi(x,y)$ and  $a\in \Mon$. Let   $p=\tp(a)$. By Proposition \ref{prop slb3} the set $\phi(\Mon,a)$ can be partitioned into  $a$-definable, definably convex pieces. Each of them is defined by a $u$-convex formula with parameter $a$, so  
$\models \phi(x,a)\leftrightarrow\theta(x,a)$ holds for some formula $\theta(x,y)$ which is a disjunction of $u$-convex formulae. By compactness, there exists $\psi_p(x)\in p$ such that
\begin{center}
$\models (\psi_p(y)\land\phi(x,y))\leftrightarrow(\psi_p(y)\land\theta(x,y))$.
 \end{center} 
Denote the formula on the right hand side of the equivalence by $\theta_p(x,y)$; it is the conjunction of
two $u$-convex formulae.
 Since  $\{[\psi_p]\mid p\in S_1(T)\}$  is a cover of $S_1(T)$, by compactness we can extract a finite subcover. Then the disjunction of the formulae $\theta_p(x,y)$ corresponding to the elements of the subcover is equivalent to $\phi(x,y)$. Clearly, this disjunction is a Boolean combination of $u$-convex formulae. 
\end{proof}

  Lemma \ref{Lem ccel orders have SLB} implies that the complete theory of any definitional expansion of a \ccel-order satisfies (SLB). By Theorem \ref{teorema2} adding names for all $u$-convex formulae to the language guarantees elimination of quantifiers. 
 
\begin{cor}\label{cor SLB elim quantifiers}
If $(M,<,P_i,E_j,R_{n,j})_{i\in I,j\in J,n\in\mathbb Z}$ is a   linearly ordered structure in which all definable unary  sets $P_i$ and convex equivalence relations $E_j$ are named, while each $R_{n,j}$ is a binary relation defined by  $x\leqslant S^{n}_{E_i}(y)$ for $n\geqslant 0$ and  $x<S^{n}_{E_i}(y)$ for $n<0$, then its complete theory eliminates quantifiers. 
\end{cor} 

\begin{proof}[{\it Proof of Theorem \ref{teorema1}}]
It suffices to show that a saturated linearly ordered structure is definitionally equivalent with its \ccel-reduct if and only if it satisfies (SLB). The right-to left direction  is a consequence of Theorem \ref{teorema2}: (SLB) implies that  every formula is equivalent with a Boolean combination of $u$-convex  formulae which are expressible in the language of the \ccel-reduct, too.  The other direction is an immediate consequence of Lemma \ref{Lem ccel orders have SLB}.
\end{proof}

\bibliographystyle{abbrv}
\bibliography{rubin_arxiv.bib}
\end{document}